\documentclass[preprint,11pt]{article}
\usepackage{amssymb}
\usepackage{amsfonts}
\usepackage{mathrsfs}
\usepackage{amssymb,amsfonts,amsmath}
\usepackage{graphicx,graphics,epstopdf}
\usepackage{url}
\graphicspath{{./figures/}}
\usepackage{latexsym}
\usepackage[]{algorithm2e}
\usepackage{esint}
\usepackage{multirow}
\usepackage{ctable} 
\newcommand{\specialcell}[2][c]{\begin{tabular}[#1]{@{}c@{}}#2\end{tabular}}
\usepackage{subfigure}
\usepackage{times}
\usepackage{multicol,enumerate}
\usepackage{epstopdf}
\usepackage{color}
\usepackage{fullpage}
\usepackage{setspace}
\usepackage[english]{babel}
\usepackage[version=3]{mhchem}
\usepackage{amsmath, amssymb, amsthm}
\usepackage{verbatim, latexsym}
\usepackage{colortbl}
\usepackage{multirow}
\usepackage[OT2,OT1]{fontenc}
\usepackage{float}
\restylefloat{table}
\newcommand\cyr
{
\renewcommand\rmdefault{wncyr}
\renewcommand\sfdefault{wncyss}
\renewcommand\encodingdefault{OT2}
\normalfont
\selectfont
}
\DeclareTextFontCommand{\textcyr}{\cyr}

\def\Xint#1{\mathchoice
{\XXint\displaystyle\textstyle{#1}}%
{\XXint\textstyle\scriptstyle{#1}}%
{\XXint\scriptstyle\scriptscriptstyle{#1}}%
{\XXint\scriptscriptstyle\scriptscriptstyle{#1}}%
\!\int}
\def\XXint#1#2#3{{\setbox0=\hbox{$#1{#2#3}{\int}$ }
\vcenter{\hbox{$#2#3$ }}\kern-.6\wd0}}

\def\dashint{\Xint-}
\newcommand{\roma}{\mathrm{I}}
\newcommand{\romb}{\mathrm{II}}
\newcommand{\romc}{\mathrm{III}}

\newcommand{\etamaxmin}[1]{\eta_{\text{#1}}}

\newcommand{\ti}{\mathrm{T}_i}

\newcommand{\dx}{\,\mathrm{d}x}
\newcommand{\prnt}[1]{\left( #1 \right)}

\newcommand{\norm}[1]{\left\|#1\right\|}

\newcommand{\seminormE}[2]{|#1|_{H^{1}_{\kappa}\prnt{#2}}}
\newcommand{\normL}[2]{\norm{#1}_{L^2\prnt{#2}}}

\newcommand{\normLii}[2]{\norm{#1}_{L^2_{\widetilde{\kappa}^{-1}}\prnt{#2}}}

\newcommand{\innerE}[2]{\langle {#1},{#2}\rangle}

\newcommand{\Cpoin}[1]{{\rm C}_{\mathrm{poin}}(#1)}

\newtheorem{theorem}{Theorem}[section]

\newtheorem{assumption}{Assumption}[section]
\newtheorem{remark}{Remark}[section]

\newtheorem{proposition}{Proposition}[section]
\numberwithin{equation}{section}
\usepackage{soul}
\setstcolor{red}

\title{Edge Multiscale Methods for elliptic problems with heterogeneous coefficients}
\author{Shubin Fu\thanks{Department of Mathematics, Chinese University of Hong Kong, Hong Kong Special Administrative Region. (\texttt{shubinfu89@gmail.com})} \and Eric Chung\thanks{Department of Mathematics, Chinese University of Hong Kong, Hong Kong Special Administrative Region. (\texttt{eric.t.chung@gmail.com})} \and Guanglian Li\thanks{Corresponding anthor. Department of Mathematics, Imperial College London, South Kensington, London SW7 2AZ,
UK. (\texttt{lotusli0707@gmail.com}, \texttt{guanglian.li@imperial.ac.uk}). GL acknowledges the
support from the Royal Society through a Newton international fellowship. Part of this work was done while EC and GL visited Erwin Schr\"{o}dinger International Institute for Mathematics and Physics (ESI, Vienna) for the research program: Numerical Analysis of Complex PDE Models in the Sciences.}
}
\begin{document}
\maketitle
\begin{abstract}
In this paper, we proposed two new types of edge multiscale methods motivated by \cite{GL18} to solve Partial Differential Equations (PDEs) with high-contrast heterogeneous coefficients: Edge spectral multiscale Finte Element method (ESMsFEM) and Wavelet-based edge multiscale Finite Element method (WEMsFEM). Their convergence rates for elliptic problems with high-contrast heterogeneous coefficients are demonstrated in terms of the coarse mesh size $H$, the number of spectral basis functions and the level of the wavelet space $\ell$, which are verified by extensive numerical tests.

\noindent{\bf Keywords:}
multiscale, heterogeneous, edge, high-contrast, Steklov eigenvalue, wavelets
\end{abstract}

\section{Introduction}
The accurate mathematical modeling of many important applications, e.g., composite materials, porous media and reservoir simulation, involves elliptic problems with heterogeneous coefficients. In order to adequately describe the intrinsic complex properties in practical scenarios, the heterogeneous coefficients can have
both multiple inseparable scales and high-contrast. Due to this disparity of scales, the classical numerical treatment becomes prohibitively expensive
and even intractable for many multiscale applications. Nonetheless, motivated by the broad spectrum of practical applications, a large number of multiscale model reduction techniques, e.g., multiscale finite element methods (MsFEMs),
heterogeneous multiscale methods (HMMs), variational multiscale methods, flux norm approach, generalized multiscale finite element methods (GMsFEMs) and localized orthogonal decomposition (LOD), have been proposed in the literature \cite{MR1455261,MR1979846,MR1660141,MR2721592, egh12, MR3246801, li2017error} over
the last few decades. They have achieved great success in the efficient and accurate simulation of heterogeneous problems. Amongst these numerical methods, the GMsFEMs \cite{egh12} have demonstrated extremely promising numerical results for a wide variety of problems, and thus they are becoming increasingly popular.

However, the mathematical understanding of GMsFEMs remains largely missing, despite numerous successful empirical evidences. Recently, the author in \cite{GL18} provided a first mathematical justification without any restrictive assumptions or oversampling technique by representing the solution restricted on each local domain as a summation of three parts, and then approximating rigorously each component by means of precalculated multiscale basis functions, namely, a specific multiscale basis function, local multiscale basis functions over the local domain and over the coarse edges. One of the critical challenges in \cite{GL18} is to make every estimate independent of the heterogeneity in the coefficient, e.g., the mulple scales and large deviation of values. As proved in \cite{GL18}, among the three types of multiscale basis functions to approximate each component of the solution over each local region, the local multiscale basis functions over the coarse edges play a critical role. Its energy error estimate poses a certain difficulty in the proof, which relies mainly on the regularity properties \cite{chu2010new,li2017low} of the high-contrast problems and the transposition method \cite{MR0350177}. In particular, the approximation property of the solution over the coarse edges determines the approximation of the solution in energy norm.

Motivated by this result, we propose two types of edge multiscale Finite Element methods in Section \ref{sec:edge} to solve PDEs with high-contrast heterogeneous coefficients: Edge Spectral Multiscale Finite Element Method (ESMsFEM) and Wavelet-based Edge Multiscale Finite Element Method (WEMsFEM). The edge spectral multiscale basis functions and the wavelets (e.g., Haar wavelets and hierarchical bases) \cite{MR1162107,Yserentant1986} are utilized to approximate the trace of the solution over each coarse edge, correspondingly. On the one hand, due to the large variations and discontinuities in the heterogeneous coefficients, this gives arise to singular behavior and the solution owns very low regularity in certain regions of the computational domain. On the other hand, the wavelets are capable of approximating functions with very low regularities and their approximation properties are reflected or characterized by the size of the finest level. Moreover, the hierarchical structure intrinsically built in the wavelets makes the wavelets excellent candidates to approximate functions with low regularities. For this reason, we apply the wavelets as the basis functions on the edges. In addition, we derive the energy error estimates for each approach and present several numerical tests in 2-dimension and 3-dimension to demonstrate the accuracy of our new proposed methods.

This work is not the first one to apply ideas from wavelets to approximating multiscale partial differential equations. The authors in \cite{engquist2002wavelet} proposed a projection-based numerical homogenization scheme which utilizes different levels of wavelet spaces as the coarse space and the fine space. In specific, this procedure involves global correction operators over the computational domain, and the wavelets are utilized to approximate the solution directly. Recently, wavelets are applied to derive an orthogonal decomposition of the solution \cite{owhadi2015bayesian}, which again, approximate the solution on the global or localized domain directly. To the best of our knowledge, this paper represents the first one, where the wavelets are introduced to approximate the trace of the solution over each coarse edge. Because of this, there is no further localization technique required in our methods. 

The remainder of the paper is arranged as below. We formulate in Section \ref{sec:prelim} the heterogeneous elliptic problem and the main idea of GMsFEMs. Then we present the basic notation and approximation properties of Haar wavelets and hierarchical bases in Section \ref{sec:wavelets}. This is then followed by Section \ref{sec:edge} dealing with two novel edge multiscale methods, which are the key findings of this paper.
Their theoretical and numerical performance are presented in Sections \ref{sec:error} and \ref{sec:numer}. Finally, we conclude the paper with several remarks in Section \ref{sec:conclusion}.

\section{Preliminaries}\label{sec:prelim}
We first formulate the heterogeneous elliptic problem to present our new multiscale methods. Let $D\subset
\mathbb{R}^d$ ($d=1,2,3$) be an open bounded Lipschitz domain {with a boundary $\partial D$}. We seek a function $u\in V:=H^{1}_{0}(D)$ such that
\begin{equation}\label{eqn:pde}
\begin{aligned}
\mathcal{L}u:=-\nabla\cdot(\kappa\nabla u)&=f &&\quad\text{ in }D,\\
u&=0 &&\quad\text{ on } \partial D,
\end{aligned}
\end{equation}
where the force term $f\in L^2(D)$ and the permeability coefficient $\kappa\in L^{\infty}(D)$ with $\alpha\leq\kappa(x)
\leq\beta$ almost everywhere for some lower bound $\alpha>0$ and upper bound $\beta>\alpha$. We denote by $\Lambda:=
\frac{\beta}{\alpha}$ the ratio of these bounds, {which reflects the contrast of the coefficient $\kappa$}. Note that
the existence of multiple scales in the coefficient $\kappa$ rends directly solving Problem \eqref{eqn:pde} challenging, since
resolving the problem to the finest scale would incur huge computational cost.

Now we present basic facts related to Problem \eqref{eqn:pde} and briefly describe the GMsFEM (and also to fix the notation).
Let the space $V:=H^{1}_{0}(D)$ be equipped with the (weighted) inner product
\begin{align*}
\innerE{v_1}{v_2}_{D}=:a(v_1,v_2):=\int_{D}\kappa\nabla v_1\cdot\nabla v_2\dx\quad \text{ for all } v_1, v_2\in V,
\end{align*}
and the associated energy norm
\begin{align*}
\seminormE{v}{D}^2:=\innerE{v}{v}_{D}\quad \text{ for all } v\in V.
\end{align*}
We denote by $W:=L^2(D)$ equipped with the usual norm $\normL{\cdot}{D}$ {and inner product $(\cdot,\cdot)_{D}$}.

The weak formulation for problem \eqref{eqn:pde} is to find $u\in V$ such that
\begin{align}\label{eqn:weakform}
a(u,v)=(f,v)_{D} \quad \text{for all
} v\in V.
\end{align}
The Lax-Milgram theorem implies the well-posedness of problem \eqref{eqn:weakform}.

To discretize problem \eqref{eqn:pde}, we first introduce fine and coarse grids.
Let $\mathcal{T}_H$ be a regular partition of the domain $D$ into
finite elements (triangles, quadrilaterals, tetrahedral, etc.) with a mesh size $H$. We refer to
this partition as coarse grids, and accordingly the coarse elements. Then each coarse element is further partitioned
into a union of connected fine grid blocks. The fine-grid partition is denoted by
$\mathcal{T}_h$ with $h$ being its mesh size. Over the fine mesh $\mathcal{T}_h$, let $V_h$ be the conforming piecewise
linear finite element space:
\[
V_h:=\{v\in \mathcal{C}: V|_{T}\in \mathcal{P}_{1} \text{ for all } T\in \mathcal{T}_h\},
\]
where $\mathcal{P}_1$ denotes the space of linear polynomials. Then the fine-scale solution $u_h\in V_h$ satisfies
\begin{align}\label{eqn:weakform_h}
a(u_h,v_h)=(f,v_h)_{D} \quad \text{ for all } v_h\in V_h.
\end{align}
The fine-scale solution $u_h$ will serve as a reference solution in Section \ref{sec:numer}. Note that due to the presence of multiple scales in the coefficient $\kappa$, the fine-scale mesh size $h$ should be commensurate with the smallest scale and thus it can be very small in order to obtain an accurate solution. This necessarily involves huge computational complexity, and more efficient methods are in great demand.


In this work, we are concerned with flow problems with high-contrast heterogeneous coefficients,
which involve multiscale permeability fields, e.g., permeability fields with vugs and faults, and
furthermore, can be parameter-dependent, e.g., viscosity. Under such scenario, the computation of the
fine-scale solution $u_h$ is vulnerable to high computational
complexity, and one has to resort to multiscale methods. The GMsFEM has been extremely successful
for solving multiscale flow problems, which we briefly recap below. 

The GMsFEM aims at solving Problem \eqref{eqn:pde} on the coarse mesh $\mathcal{T}_{H}$
cheaply, which, meanwhile, maintains a certain accuracy compared to the fine-scale solution $u_h$. To describe the
GMsFEM, we need a few notation. The vertices of $\mathcal{T}_H$
are denoted by $\{O_i\}_{i=1}^{N}$, with $N$ being the total number of coarse nodes.
The coarse neighborhood associated with the node $O_i$ is denoted by
\begin{equation} \label{neighborhood}
\omega_i:=\bigcup\{ K_j\in\mathcal{T}_H: ~~~ O_i\in \overline{K}_j\}.
\end{equation}
We refer to Figure~\ref{schematic} for an illustration of neighborhoods and elements subordinated to the coarse
discretization $\mathcal{T}_H$. Throughout, we use $\omega_i$ to denote a coarse neighborhood.

\begin{figure}[htb]
  \centering
  \includegraphics[width=0.65\textwidth]{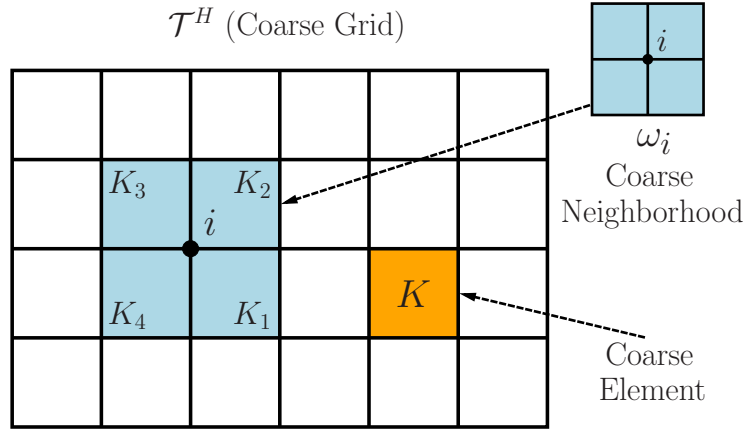}
  \caption{Illustration of a coarse neighborhood and coarse element.}
  \label{schematic}
\end{figure}

Next, we outline the GMsFEM with a conforming Galerkin (CG) formulation. Let $1\leq i\leq N$ be a certain coarse node. Note that $\omega_i$ is
the support of the multiscale basis functions to be identified, and $\ell_i\in \mathbb{N}_{+}$ is the number of those multiscale basis functions associated with $\omega_i$. They are denoted as $\psi_k^{\omega_i}$ for $k=1,\cdots,\ell_i$.  Throughout, the subscript $i$ denotes the $i$-th coarse node or coarse neighborhood.
Generally, the GMsFEM utilizes multiple basis functions per coarse neighborhood $\omega_i$,
and the index $k$ represents the numbering of these basis functions.
In turn, the CG multiscale solution $u_{\text{ms}}$ is sought as $u_{\text{ms}}=\sum\limits_{i=1}^{N}\sum\limits_{k=1}^{\ell_i} c_{k}^i \psi_{k}^{\omega_i}$.
Once the basis functions $\psi_k^{\omega_i}$ are identified, the CG global coupling is given through the variational form
\begin{equation}\label{eq:multiscale}
a(u_{\text{ms}},v)=(f,v)_{D}, \quad \text{for all} \, \, v\in
V_{\text{off}},
\end{equation}
where $V_{\text{off}}$ denotes the multiscale space spanned by these multiscale basis functions.

We conclude the section with the following assumption on the computational domain $D$ and the heterogeneous coefficient $\kappa$.
\begin{assumption}[Structure of $D$ and $\kappa$]\label{ass:coeff}
Let $D$ be a domain with a $C^{1,\alpha}$ $(0<\alpha<1)$ boundary $\partial D$,
and $\{D_i\}_{i=1}^m\subset D$ be $m$ pairwise disjoint strictly convex open subsets, {each with a $C^{1,\alpha}$ boundary
$\Gamma_i:=\partial D_i$,} and denote $D_0=D\backslash \overline{\cup_{i=1}^{m} D_i}$. 
Let the permeability coefficient $\kappa$ be piecewise regular function defined by
\begin{equation}
\kappa=\left\{
\begin{aligned}
&\eta_{i}(x) &\text{ in } D_{i},\\
&1 &\text{ in }D_0.
\end{aligned}
\right.
\end{equation}
Here $\eta_i\in C^{\mu}(\bar{D_i})$ with $\mu\in (0,1)$ for $i=1,\cdots,m$. Denote  $\etamaxmin{min}:=\min_{i}\{\eta_i\}\geq 1$ and $\etamaxmin{max}:=\max_{i}\{\eta_i\}$.
\end{assumption}

Under Assumption \ref{ass:coeff}, the coefficient $\kappa$ is $\Gamma$-{\em quasi-monotone} on each coarse neighborhood $\omega_i$ and the global domain $D$
(see \cite[Definition 2.6]{pechstein2012weighted} for the precise definition) with either $\Gamma:=\partial \omega_i$
or $\Gamma:=\partial D$. Then the following weighted Friedrichs inequality \cite[Theorem 2.7]{pechstein2012weighted} holds.
\begin{theorem}[Weighted Friedrichs inequality]\label{thm:friedrichs}
Let $\text{diam}(D)$ be the diameter of the bounded domain $D$ and $\omega_i\subset D$. Define
\begin{align}
\Cpoin{\omega_i}&:=H^{-2}\max\limits_{w\in H^1_0(\omega_i)}\frac{\int_{\omega_i}{\kappa}w^2\dx}{\int_{\omega_i}\kappa|\nabla w|^2\dx},\label{eq:poinConstant}\\
\Cpoin{D}&:=\text{diam}(D)^{-2}\max\limits_{w\in H^1_0(D)}\frac{\int_{D}{\kappa}w^2\dx}{\int_{D}\kappa|\nabla w|^2\dx}.\label{eq:poinConstantG}
\end{align}
Then the positive constants $\Cpoin{\omega_i}$ and $\Cpoin{D}$ are independent of the contrast of $\kappa$.
\end{theorem}
\begin{remark}
Below we only require that the constants $\Cpoin{\omega_i}$ and $\Cpoin{D}$ be independent
of the contrast in $\kappa$. Assumption \ref{ass:coeff} is one sufficient condition to ensure this,
and it can be relaxed \cite{pechstein2012weighted}.
\end{remark}

\section{Hierarchical subspace splitting over $I=:[0,1]$}\label{sec:wavelets}
In this section, we introduce two types of wavelets on the unit interval $I:=[0,1]$: Haar wavelets and hierarchical bases . They facilitate hierarchically splitting the space $L^2(I)$.
\subsection{Haar wavelets}
Let the level parameter and the mesh size be $\ell$ and $h_{\ell}:=2^{-\ell}$ with $\ell\in \mathbb{N}$, respectively. Then the grid points on level $\ell$ are
\[
x_{\ell,j}=j\times h_{\ell},\qquad 0\leq j\leq 2^{\ell}.
\]
Let the scaling function $\phi(x)$ and the mother wavelet $\psi(x)$ be given by
\begin{equation*}
\phi(x)=
\left\{
\begin{aligned}
&1, &&\text{ if } 0\leq x\leq 1,\\
&0, &&\text{ otherwise,}
\end{aligned}
\right.
\qquad
\psi(x)=
\left\{
\begin{aligned}
&1, &&\text{ if } 0\leq x\leq 1/2,\\
&-1, &&\text{ if }1/2< x\leq 1,\\
&0, &&\text{ otherwise}.
\end{aligned}
\right.
\end{equation*}
By means of dilation and translation, the mother wavelet $\psi(x)$ can result in orthogonal decomposition of the space $L^2(I)$. To this end, we can define the basis functions on level $\ell$ by
\begin{align*}
\psi_{\ell,j}^{\roma}(x):=2^{\frac{\ell-1}{2}}\psi(2^{\ell-1}x-j) \quad \text{ for all }\quad 0\leq j\leq 2^{\ell-1}-1.
\end{align*}
The subspace of level $\ell$ is
\[
W_{\ell}^{\roma}:=\text{span}\{\psi_{\ell,j}^{\roma}:\quad 0\leq j\leq 2^{\ell-1}-1\}.
\]
Note that the subspace $W_{\ell}^{\roma}$ is orthogonal to $W_{\ell'}^{\roma}$ in $L^2(I)$ for any two different levels $\ell\neq \ell'$. We denote by $V_{\ell}^{\roma}$ as the subspace in $L^2(I)$ up to level $\ell$, which is defined by
\[
V_{\ell}^{\roma}:=\oplus_{m\leq\ell}W_{m}^{\roma}.
\]
Due to the orthogonality of the subspaces $W_{\ell}$ on different levels, there holds
\[
V_{\ell+1}^{\roma}=V_{\ell}^{\roma}\oplus_{L^2(I)} W_{\ell+1}^{\roma}.
\]
Consequently, it yields the hierarchical structure of the subspace $V_{\ell}$, namely,
\[
V_{0}^{\roma}\subset V_{1}^{\roma}\subset \cdots\subset V_{\ell}^{\roma}\subset V_{\ell+1}^{\roma}\cdots
\]
Furthermore, the following orthogonal decomposition of the space $L^2(I)$ holds
\[
L^2(I)=\oplus_{\ell}W_{\ell}^{\roma}.
\]
\subsection{Hierarchical bases }
Let the level parameter and the mesh size be $\ell$ and $h_{\ell}:=2^{-\ell}$ with $\ell\in \mathbb{N}$, respectively. Then the grid points on level $\ell$ are
\[
x_{\ell,j}=j\times h_{\ell},\qquad 0\leq j\leq 2^{\ell}.
\]

We can define the basis functions on level $\ell$ by
\begin{equation*}
\psi_{\ell,j}^{\romb}(x)=
\left\{
\begin{aligned}
&1-|x/h_{\ell}-j|, &&\text{ if }  x\in [(j-1)h_{\ell},(j+1)h_{\ell}]\cap [0,1],\\
&0, &&\text{ otherwise.}
\end{aligned}
\right.
\end{equation*}
Define the set on each level $\ell$ by
\begin{equation*}
B_{\ell}:=\Bigg\{
j\in\mathbb{N}\Bigg|
\begin{aligned}
&j=1,\cdots,2^{\ell}-1, j \,\rm{ is\, odd }, &&\rm{if }\,\ell>0\\
&j=0,1,&&\rm{ if }\,\ell=0
\end{aligned}
\Bigg\}.
\end{equation*}
The subspace of level $\ell$ is
\[
W_{\ell}^{\romb}:=\text{span}\{\psi_{\ell,j}^{\romb}:\quad j\in B_{\ell}\}.
\]
We denote $V_{\ell}$ as the subspace in $L^2(I)$ up to level $\ell$, which is defined by the direct sum of subspaces
\[
V_{\ell}^{\romb}:=\oplus_{m\leq\ell}W_{m}^{\romb}.
\]
Consequently, this yields the hierarchical structure of the subspace $V_{\ell}$, namely,
\[
V_{0}^{\romb}\subset V_{1}^{\romb}\subset \cdots\subset V_{\ell}^{\romb}\subset V_{\ell+1}^{\romb}\cdots
\]
Furthermore, the following hierarchical decomposition of the space $L^2(I)$ holds
\[
L^2(I)=\lim_{\ell\to\infty}\oplus_{m\leq\ell}W_{m}^{\romb}.
\]
Note that one can derive the hierarchical decomposition of the space $L^2(I^d)$ for $d>1$ by means of tensor product. Note further that we will use the subspace $V_{\ell}^{\roma}$ and $V_{\ell}^{\romb}$ to approximate the restriction of the exact solution $u$ on each coarse edge.

In this paper, we will only focus on the convergence analysis of multiscale algorithms, cf. Algorithm \ref{algorithm:wavelet}, based upon the Haar wavelets $V_{\ell}^{\roma}$. The convergence analysis of multiscale algorithms based upon the hierarchical bases  $V_{\ell}^{\romb}$ can be derived similarly.

Throughout this paper, $(\cdot,\cdot)_{T}$ for some domain $T\subset D$ or some edges $T\subset \partial \omega_i$ denotes the inner product in the Hilbert space $L^2(T)$. We use $A\lesssim B$ if $A\leq CB$ for some benign constant that is independent of the multiple scales and high contrast in the coefficient $\kappa$ and the coarse scale mesh size $H$. 

\begin{proposition}[Approximation properties of the hierarchical space $V_{\ell}^{\roma}$]\label{prop:approx-wavelets}
Let $P_{\ell}$ be $L^2(I)$-orthogonal projection onto $V_{\ell}^{\roma}$ for each level $\ell\geq 0$ and let $s>0$. Then there holds
\begin{align*}
P_{\ell+1} v&=P_{\ell}v+\sum\limits_{j=0}^{2^{\ell}-1}(v,\psi_{\ell+1,j}^{\roma})_{I}\psi_{\ell+1,j}^{\roma}&\text{for all }v\in L^2(I)\\
\|v-P_{\ell}v\|_{L^2(I)}&\lesssim 2^{-s\ell}|v|_{H^s(I)} &\text{for all }v\in H^s(I).
\end{align*}
\end{proposition}
\begin{proof}
The first assertion can be found in \cite{MR1162107}. To prove the second assertion, define the operator
\[
\mathcal{T}: H^s(I)\to L^2(I) \quad\text{ by }\mathcal{T}v:=v-P_{\ell}v.
\]
Let $s:=0$. Then the $L^2(I)$-orthogonality of $P_{\ell}$ implies
\[
\|\mathcal{T}v\|_{L^2(I)}:=\|v-P_{\ell}v\|_{L^2(I)}\leq\|v\|_{L^2(I)} \quad\text{for all }v\in L^2(I).
\]
Furthermore, let $s:=1$. Since the residual $v-P_{\ell}v$ is orthogonal to $V_{\ell}^{\roma}$. Therefore, we obtain
\[
\int_{j\times 2^{-\ell}} ^{(j+1)\times 2^{-\ell}}(v-P_{\ell}v)\dx=0 \text{ for all } j=0,\cdots 2^{\ell}-1.
\]
Consequently, for all $v\in H^1(I)$, the Poincar\'{e} inequality leads to
\begin{align*}
\|\mathcal{T}v\|_{L^2(I)}^2
&:=\|v-P_{\ell}v\|_{L^2(I)}^2
=\sum\limits_{j=0}^{2^{\ell}}\int_{j\times 2^{-\ell}} ^{(j+1)\times 2^{-\ell}}|v-P_{\ell}v|^2\dx\\
&\lesssim 2^{-2\ell}|v|_{H^1(I)}^2.
\end{align*}
Taking the square root on both sides gives
\begin{align*}
\|\mathcal{T}v\|_{L^2(I)}\lesssim 2^{-2\ell}|v|_{H^1(I)} \quad\text{ for all }v\in H^1(I).
\end{align*}
Finally, the preceding two estimates, together with the interpolation theory, prove the second assertion.
\end{proof}
\section{Edge multiscale methods}\label{sec:edge}

We propose in this section two new multiscale methods based on GMsFEMs. In specific, their multiscale basis functions are defined locally on each coarse neighborhood independently, and thereby they can be calculated in parallel. The first multiscale method utilizes the dominant eigenvectors from the local Stechlov eigenvalue problem as the local multiscale basis functions. The second one uses wavelets to approximate the solution restricted on each coarse edge. To obtain conforming global basis functions, we utilize the Partition of Unity finite element method \cite{melenk1996partition,EFENDIEV2011937}. Its main idea is to seek local multiscale basis functions in each coarse neighborhood which own certain approximation properties to the exact solution restricted on each coarse neighborhood, and to use the fact that the global multiscale basis functions obtained from those local multiscale basis functions by the partition of unity functions inherit these approximation properties.

To this end, we begin with an initial coarse space $V^{\text{init}}_0 = \text{span}\{ \chi_i \}_{i=1}^{N}$.
The functions $\chi_i$ are the standard multiscale basis functions on each coarse element $K\in \mathcal{T}_{H}$ defined by
\begin{alignat}{2} \label{pou}
-\nabla\cdot(\kappa(x)\nabla\chi_i) &= 0  &&\quad\text{ in }\;\;K, \\
\chi_i &= g_i &&\quad\text{ on }\partial K, \nonumber
\end{alignat}
where $g_i$ is affine over $\partial K$ with $g_i(O_j)=\delta_{ij}$ for all $i,j=1,\cdots, N$. Recall that $\{O_j\}_{j=1}^{N}$ are the set of coarse nodes on $\mathcal{T}_{H}$. Next we define the weighted coefficient:
\begin{equation}\label{defn:tildeKappa}
\widetilde{\kappa} =H^2 \kappa \sum_{i=1}^{N}  | \nabla \chi_i |^2.
\end{equation}
Furthermore, let $\widetilde{\kappa}^{-1}$ be defined by
\begin{equation}\label{eq:inv-tildeKappa}
\widetilde{\kappa}^{-1}(x)=
\left\{
\begin{aligned}
&\widetilde{\kappa}^{-1}, \quad &&\text{ when } \widetilde{\kappa}(x)\ne 0\\
&1, \quad &&\text{ otherwise }.
\end{aligned}
\right.
\end{equation}
The weighted $L^2(D)$ space is
\begin{align*}
L_{{\kappa}^{-1}}^{2}(D):=&\Big\{w:\|w\|_{L^2_{\kappa^{-1}(D)}}^2:=\int_{D}{\kappa}^{-1} w^2\dx<\infty  \Big\}.
\end{align*}
Similarly, we define the following weighted Sobolev spaces with their associated norms: $(L_{\widetilde{\kappa}^{-1}}^{2}(D),\|\cdot\|_{L^2_{\widetilde{\kappa}^{-1}(D)}})$.

Note that there are many other alternatives for the partition of unity functions besides using the multiscale basis functions \eqref{pou}, e.g., one can utilize the flat top type of partition of unity functions proposed in \cite{Griebel.Schweitzer:2000}.

\subsection{Edge Spectral Multiscale Finte Element Method}\label{subsec:edgeSpectral}

The first new multiscale method, coined as the edge spectral multiscale method (ESMsFEM), is inspired by the recent results derived in \cite{GL18}. To obtain a good approximation space for the solution $u$ in \eqref{eqn:pde}, one only needs to derive a good local approximation space on each coarse neighborhood $\omega_i$ to $u|_{\omega_i}$ according to the main theory of Partition of Unity Finite Element Method \cite{melenk1996partition}. In that paper, the restriction $u|_{\omega_i}$ is split into three components, each of which is approximated by local  multiscale basis functions with proved convergence rate. Since one of the components is of $\mathcal{O}(H)$, this part is negligible and is removed from Algorithm \ref{algorithm:spectral}.
\begin{algorithm}[H]
\begin{tabular}{l l}
\specialrule{.2em}{.1em}{.1em}
\textbf{Input}:&Coarse neighborhood $\omega_i$ and its total number $N$; the number of multiscale basis functions \\&
$\ell_i\in \mathbb{N}_{+}$; the partition of unity function $\chi_i$;\\
\textbf{Output}:& Multiscale solution $u_{\text{ms}}^{\text{ES}}$.\\
\specialrule{.2em}{.1em}{.1em}
1.& Solve for the Steklov eigenvalue problem and reorder the eigenvalues non-decreasingly. \\
&Seek $(\lambda_{j}^{\ti}, v_{j}^{\ti})\in \mathbb{R}\times H^1_{\kappa}(\omega_i)$ such that\\&
$
\begin{cases}
-\nabla\cdot(\kappa\nabla v_{j}^{\ti})  = 0 &\quad\text{in} \, \, \, \omega_i,\\
\kappa\frac{\partial}{\partial n}v_{j}^{\ti}=\lambda_{j}^{\ti} v_{j}^{\ti}&\quad\text{ on } \partial \omega_i.
\end{cases}
$\\
2.& Solve one local problem.\\&
$
\left\{
\begin{aligned}
-\nabla\cdot(\kappa\nabla v^{i})&=\frac{\widetilde{\kappa}}{\int_{\omega_i}\widetilde{\kappa}\dx} \quad&&\text{ in } \omega_i,\\
-\kappa\frac{\partial v^{i}}{\partial n}&=|\partial\omega_i|^{-1}\quad&&\text{ on }\partial \omega_i.
\end{aligned}
\right.
$
\\
3.& Build global multiscale space.\\&
$
V_{\text{off}}^{\rm ES}  := \text{span} \{\chi_i v^i,\chi_i v_{k}: \,  \, 1 \leq i \leq N,\,\,\, 1 \leq k \leq \ell_i-1\}.
$\\
4. &Solve for \eqref{eq:multiscale} by Conforming Galerkin method in $V_{\text{off}}^{\rm ES} $ to obtain  $u_{\text{ms}}^{ES}$.
\\[0cm]
\specialrule{.2em}{.1em}{.1em}
\end{tabular}
\caption{Edge Spectral Multiscale Finte Element Method (ESMsFEM)}
\label{algorithm:spectral}
\end{algorithm}
Algorithm \ref{algorithm:spectral} proceeds as follows. Recall that $N$ denotes the total number of coarse nodes in the coarse mesh $\mathcal{T}_{H}$ and $\omega_i$ is the coarse neighborhood for the ith coarse node. $\chi_i$ is the partition of unity function defined on $\omega_i$, cf. \eqref{pou}. Let $1<\ell_i\in \mathbb{N}_{+}$ be the number of local multiscale basis functions on this coarse neighborhood $\omega_i$. Among them, the first $\ell_i-1$ are the dominant modes from the local Steklov eigenvalue problem, cf. Step 1. The last one arises from one specific local solver defined in Step 2. In Step 3, the global multiscale space $V_{\text{ms}}^{\text{ES}}$ is defined with the help of the partition of unity functions $\{\chi_i\}_{i=1}^{N}$. Then the multiscale solution $u_{\text{ms}}^{\text{ES}}$ is obtained by solving  \eqref{eq:multiscale} in the global multiscale space $V_{\text{ms}}^{\text{ES}}$.

\subsection{Wavelet-based Edge Multiscale Finite Element Method}\label{subsec:wavelet-based edge}

Motivated by \cite{GL18}, the local multiscale basis functions restricted on $\partial\omega_i$, which can approximate $u|_{\partial\omega_i}$ plays a vital role in approximating the solution $u\in V$ in \eqref{eqn:pde} efficiently. In view that $u|_{\partial\omega_i}\in H^{s}(\partial\omega_i)$ for some positive constant $s\geq 1/2$ and the approximation properties of the Haar wavelets, cf. Proposition \ref{prop:approx-wavelets}, the $\kappa$-harmonic functions with the Haar wavelets as the local boundary conditions lend themselves to excellent candidates for the local multiscale basis functions. Combining with the Partition of Unity Finite Element Methods and the conforming Galerkin approximation, this results in a new multiscale method, which is named as the Wavelet-based Edge Multiscale Finite Element Method (WEMsFEM), cf. Algorithm \ref{algorithm:wavelet}.
\begin{algorithm}[H]
\caption{Wavelet-based Edge Multiscale Finte Element Method (WEMsFEM)}
\label{algorithm:wavelet}
 \begin{tabular}{l l}
\specialrule{.2em}{.1em}{.1em}
\textbf{Input}:&The level parameter $\ell\in \mathbb{N}$; coarse neighborhood $\omega_i$ and its four coarse edges $\Gamma_{i,k}$ with \\&
    $k=1,2,3,4$, i.e., $\cup_{k=1}^{4}\Gamma_{i,k}=\partial\omega_i$;
    the subspace $V_{\ell,k}\subset L^2(\Gamma_{i,k})$ up to level $\ell$ on each\\& coarse edge $\Gamma_{i,k}$;
    \\
    \textbf{Output}:& Multiscale solution $u_{\text{ms},\ell}^{\text{EW}}$.\\
\specialrule{.2em}{.1em}{.1em}
1. & Denote
    $V_{i,\ell}:=\oplus_{k=1}^{4}V_{\ell,k}.$
    Then the number of basis functions in $V_{i,\ell}$ is $4\times 2^{\ell}=2^{\ell+2}$. \\&
    Denote these basis
    functions as $v_k$ for $k=1,\cdots, 2^{\ell+2}$.\\
2. &Calculate local multiscale basis $\mathcal{L}^{-1}_i (v_k)$ for all $k=1,\cdots,2^{\ell+2}$.\\
&Here, $\mathcal{L}^{-1}_i (v_k):=v$ satisfies:\\
& $
  \left\{ \begin{aligned}
          \mathcal{L}_i v&:=-\nabla\cdot(\kappa\nabla v)=0&& \mbox{in }\omega_i,\\
          v&=v_k&& \mbox{on }\partial\omega_i.
  \end{aligned}\right.
$\\
3. &Build global multiscale space. \\&
$
V_{\text{off}}^{\rm EW}  := \text{span} \{\chi_i\mathcal{L}^{-1}_i(v_k),\chi_i v^{i} : \,  \, 1 \leq i \leq N,\,\,\, 1 \leq k \leq  2^{\ell+2}\}.
$\\
4. &Solve for \eqref{eq:multiscale} by Conforming Galerkin method in
$V_{\text{off},\ell}^{\rm EW}$ to obtain $u_{\text{ms},\ell}^{\text{EW}}$.\\[0cm]
\specialrule{.2em}{.1em}{.1em}
 \end{tabular}
 \end{algorithm}
Algorithm \ref{algorithm:wavelet} proceeds as follows. As in Algorithm \ref{algorithm:spectral}, we first construct the local multiscale basis functions on each coarse neighborhood $\omega_i$. Given level parameter $\ell\in \mathbb{N}$, and the four coarse edges $\Gamma_{i,k}$ with $k=1,2,3,4$, i.e., $\cup_{k=1}^{4}\Gamma_{i,k}=\partial\omega_i$, let $V_{i,k}$ be either the hierarchical bases or Haar wavelets up to level $\ell$ on the coarse edge $\Gamma_{i,k}$. Note that we will drop the superscript for the subspaces $V_{i,j}^{\roma}$ and $V_{i,j}^{\romb}$. Let $V_{i,\ell}:=\oplus_{k=1}^{4}V_{\ell,k}$ be the edge basis functions on $\partial\omega_i$. Then $V_{i,\ell}$ becomes a good approximation space of dimension $2^{\ell+2}$ to the trace of the solution over $\partial \omega_i$, i.e., $u|_{\partial\omega_i}$.

Subsequently, we calculate the $\kappa$-harmonic functions on each coarse neighborhood $\omega_i$ with all possible Dirichlet boundary conditions in $V_{i,\ell}$, and denote the resulting local multiscale space as $\mathcal{L}^{-1}_i (V_{i,\ell})$ in Step 2. Analogous to Algorithm \ref{algorithm:spectral}, we can then define the global multiscale space as $V_{\text{off}}^{\text{EW}}$ and obtain the multiscale solution $u_{\text{ms},\ell}^{\text{EW}}$ in Steps 3 and 4.
\begin{remark}[WEMsFEM is an extension of MsFEM proposed in \cite{MR1455261}]
Let $\ell:=0$, then the local multiscale basis functions $\mathcal{L}_i^{-1}(V_{i,1})$generated in Step 2 in Algorithm \ref{algorithm:wavelet} contain the constant basis function. Since the multiscale basis functions proposed in \cite{MR1455261} serve as the partition of unity functions, cf. \eqref{pou}, Step 3 in Algorithm \ref{algorithm:wavelet} implies that $\{\chi_i\}_{i=1}^{N}\subset V_{\text{off}}^{\text{EW}}$. Consequently, our proposed WEMsFEM is one enrichment of the classical multiscale method (MsFEM).
\end{remark}

\begin{remark}
Let $n$ be the number of fine elements in each coarse element, respectively. For the sake of simplicity, we can take $n:=2^{m}$ for some positive constant $m\in \mathbb{N}$. However, this is not mandatory since we can always use interpolation operator to connect the fine grids $\mathcal{T}_h$ with the mesh grids $h_{\ell}$ for Haar wavelets or hierarchical bases of level $\ell$.
\end{remark}

\begin{remark}[Flexibility of the Wavelet-based edge multiscale basis functions]
The Wavelet-based edge multiscale basis functions can be potentially extended to more general PDEs with heterogeneous coefficients since the only modification is to replace the local operator $\mathcal{L}_i$ in Algorithm \ref{algorithm:wavelet} with the localized PDEs.
\end{remark}
\section{Error estimate}\label{sec:error}
This section is concerned with deriving the convergence rates of Algorithm \ref{algorithm:spectral} and Algorithm \ref{algorithm:wavelet} for elliptic problems with heterogeneous high-contrast coefficients, cf. Problem \eqref{eqn:pde}.

As mentioned earlier, we will focus only on the Haar wavelets in Algorithm \ref{algorithm:wavelet} since the convergence of Algorithm \ref{algorithm:wavelet} using Hierarchical bases can be obtained in a similar manner. In the following, we define the $L^2(\partial\omega_i)$-orthogonal projection $\mathcal{P}_{i,\ell}$ onto the local multiscale space up to level $\ell$: $L^2(\partial\omega_i)\to V_{i,\ell}$ by
\begin{align}\label{eq:projectionEDGE}
\mathcal{P}_{i,\ell}(v):=\sum\limits_{k=0}^{\ell}\sum\limits_{j=1}^{2^{k+2}}(v,\psi_{k,j})_{\partial\omega_i}\mathcal{L}_{i}^{-1}(\psi_{k,j}) \quad \text{ for all }v\in L^2(\partial\omega_i).
\end{align}
Here, we denote $\psi_{k,j}$ for $j=1,\cdots 2^{k+2}$ as the Haar wavelets defined on the four edges of $\omega_i$ of level $k$ and the local operator $\mathcal{L}_i$ is defined as in Algorithm \ref{algorithm:wavelet}.

The convergence of the edge spectral basis functions is a direct consequence of the results in \cite{GL18}. One main observation in \cite{GL18} is that the edge spectral basis functions play the critical role in the convergence analysis should the convergence rate of $\mathcal{O}(H)$ be after.
\begin{proposition}[Error estimate for Algorithm \ref{algorithm:spectral} to Problems \eqref{eqn:pde}]\label{prop:edgespectral}
Let Assumption \ref{ass:coeff} hold. Assume that $f\in L^2_{\widetilde{\kappa}^{-1}}(D)\cap L^2_{{\kappa}^{-1}}(D)$ and let $\ell_i\in \mathbb{N}_{+}$ for all $i=1,2,\cdots, N$. Let $u\in V$ be the solutions to Problems \eqref{eqn:pde}. There holds
\begin{equation}\label{eq:spectErr}
\begin{aligned}
\seminormE{u-u_{{\rm ms}}^{{\rm ES}}}{D}&:=\min\limits_{w\in  V_{{\rm off}}^{{\rm ES}}}\seminormE{u-w}{D}\\
&\lesssim H\normLii{f}{D}+\max_{i=1,\cdots,N}
\big\{({H^2\lambda_{\ell_i}^{\ti}})^{-\frac12}\big\}
\|f\|_{L^2_{\kappa^{-1}}(D)}.
\end{aligned}
\end{equation}
\end{proposition}
\begin{proof}
This result follows from Proposition 4.1 and Remark 4.2 in \cite{GL18}.
\end{proof}

To prove the convergence of the wavelet-based edge multiscale basis functions in Subsection \ref{subsec:wavelet-based edge}, we will first recap several results from \cite{GL18}. The solution $u$ satisfies the following equation
\begin{equation*}
\left\{
\begin{aligned}
-\nabla\cdot(\kappa\nabla u)&=f \quad&&\text{ in } \omega_i,\\
-\kappa\frac{\partial u}{\partial n}&=-\kappa\frac{\partial u}{\partial n}\quad&&\text{ on }\partial \omega_i,
\end{aligned}
\right.
\end{equation*}
which can be split into three parts, namely
\begin{align}\label{eq:decomp}
u|_{\omega_i}=u^{i,\roma}+u^{i,\romb}+u^{i,\romc}.
\end{align}
Here, the three components $u^{i,\roma}$, $u^{i,\romb}$, and $u^{i,\romc}$ are respectively given by
\begin{equation}\label{eq:u-roma1}
\left\{
\begin{aligned}
-\nabla\cdot(\kappa\nabla u^{i,\roma})&=f-\bar{f}_i \quad&&\text{ in } \omega_i\\
-\kappa\frac{\partial u^{i,\roma}}{\partial n}&=0\quad&&\text{ on }\partial \omega_i,
\end{aligned}
\right.
\end{equation}
where $\bar{f}_i:=\int_{\omega_i}f\dx\times\frac{\widetilde{\kappa}}{\int_{\omega_i}\widetilde{\kappa}\dx}$,
\begin{equation*}
\left\{
\begin{aligned}
-\nabla\cdot(\kappa\nabla u^{i,\romb})&=0 \quad&&\text{ in } \omega_i\\
-\kappa\frac{\partial u^{i,\romb}}{\partial n}&=\kappa\frac{\partial u}{\partial n}-\dashint_{\partial\omega_i}\kappa\frac{\partial u}{\partial n}\quad&&\text{ on }\partial \omega_i,
\end{aligned}
\right.
\end{equation*}
and
\[
u^{i,\romc}=v^{i}\int_{\omega_i}f\dx
\]
with $v^i$ being defined in Algorithm \ref{algorithm:spectral}. Clearly,
$u^{i,\romc}$ involves only one local solver.

In the following, we will approximate the second component $u^{i,\romb}$ by the wavelet-based edge multiscale basis functions. On the one hand, notice that $u^{i,\romb}\in H^{1/2}(\partial\omega_i)$, then we can obtain by Proposition \ref{prop:approx-wavelets} combining with a scaling argument, that
\begin{align*}
\|u^{i,\romb}-\mathcal{P}_{i,\ell}u^{i,\romb}\|_{L^2(\omega_i)}
\lesssim \sqrt{H}2^{-\ell/2}\|u^{i,\romb}\|_{H^{1/2}(\partial\omega_i)}.
\end{align*}
On the other hand, since $\kappa\geq 1$, then an application of Theorem \ref{thm:friedrichs} and the Trace inequality leads to
\begin{align*}
\|u^{i,\romb}\|_{H^{1/2}(\partial\omega_i)}\lesssim \|u^{i,\romb}\|_{H^{1}(\omega_i)}\lesssim \seminormE{u^{i,\romb}}{\omega_i}.
\end{align*}
Plugging this estimate into the previous one results in
\begin{align}\label{eq:l2edge}
\|u^{i,\romb}-\mathcal{P}_{i,\ell}u^{i,\romb}\|_{L^2(\omega_i)}
\lesssim\sqrt{H}2^{-\ell/2}\seminormE{u^{i,\romb}}{\omega_i}.
\end{align}
Now we are ready to present the estimate for the second algorithm:
\begin{proposition}[Error estimate for Algorithm \ref{algorithm:wavelet} to Problems \eqref{eqn:pde}]\label{prop:wavelet-based}
Let Assumption \ref{ass:coeff} hold. Assume that $f\in L^2_{\widetilde{\kappa}^{-1}}(D)\cap L^2_{{\kappa}^{-1}}(D)$ and let $\ell\in \mathbb{N}_{+}$. Let $u\in V$ be the solutions to Problems \eqref{eqn:pde}. There holds
\begin{equation}\label{eq:waveletErr}
\begin{aligned}
\seminormE{u-u_{{\rm ms,\ell}}^{{\rm EW}}}{D}&:=\min\limits_{w\in  V_{{\rm off}}^{{\rm EW}}}\seminormE{u-w}{D}\\
&\lesssim H\normLii{f}{D}+\sqrt{H}2^{-\ell/2}
\|f\|_{L^2_{\kappa^{-1}}(D)}.
\end{aligned}
\end{equation}
\end{proposition}
\begin{proof}
The result can be obtained by the trace approximation \eqref{eq:l2edge}, \cite[Lemma 4.3, Lemma4.4 and Proposition 4.2]{GL18}.
\end{proof}
\begin{remark}
We can infer from Proposition \ref{prop:wavelet-based} that the energy error can be bounded above by $\mathcal{O}(H)$, should the number of level $\ell:=\lceil-\log_{2}H\rceil$ become. Nevertheless, one can observe from the numerical tests that this algorithm actually is much more accurate than we have proved.
\end{remark}

\section{Numerical tests}\label{sec:numer}
In this section, we present several numerical tests to demonstrate the accuracy of our proposed methods.
In specific, we apply the multiscale algorithms ESMsFEM and WEMsFEM to solve the heterogeneous elliptic problem \eqref{eqn:pde}.

In our experiments, we take the computational domain $D:=[0,1]^d$, for $d=2$ and $3$ and the constant force is employed, namely $f:=1$. Let $\mathcal{T}_{H}$ be a regular quasi-uniform rectangular mesh over $D$ with maximal mesh size $H$ and let $\mathcal{T}_{h}$ be a regular quasi-uniform rectangular mesh over each coarse element $T\in \mathcal{T}_{H}$ with maximal mesh size $h$. Let $h:=\sqrt{2}\times 2^{-9}$.

\begin{figure}[H]
	\centering
	\subfigure[Model 1]{
		\includegraphics[trim={5cm 0 5cm 0},clip,height=3.5cm]{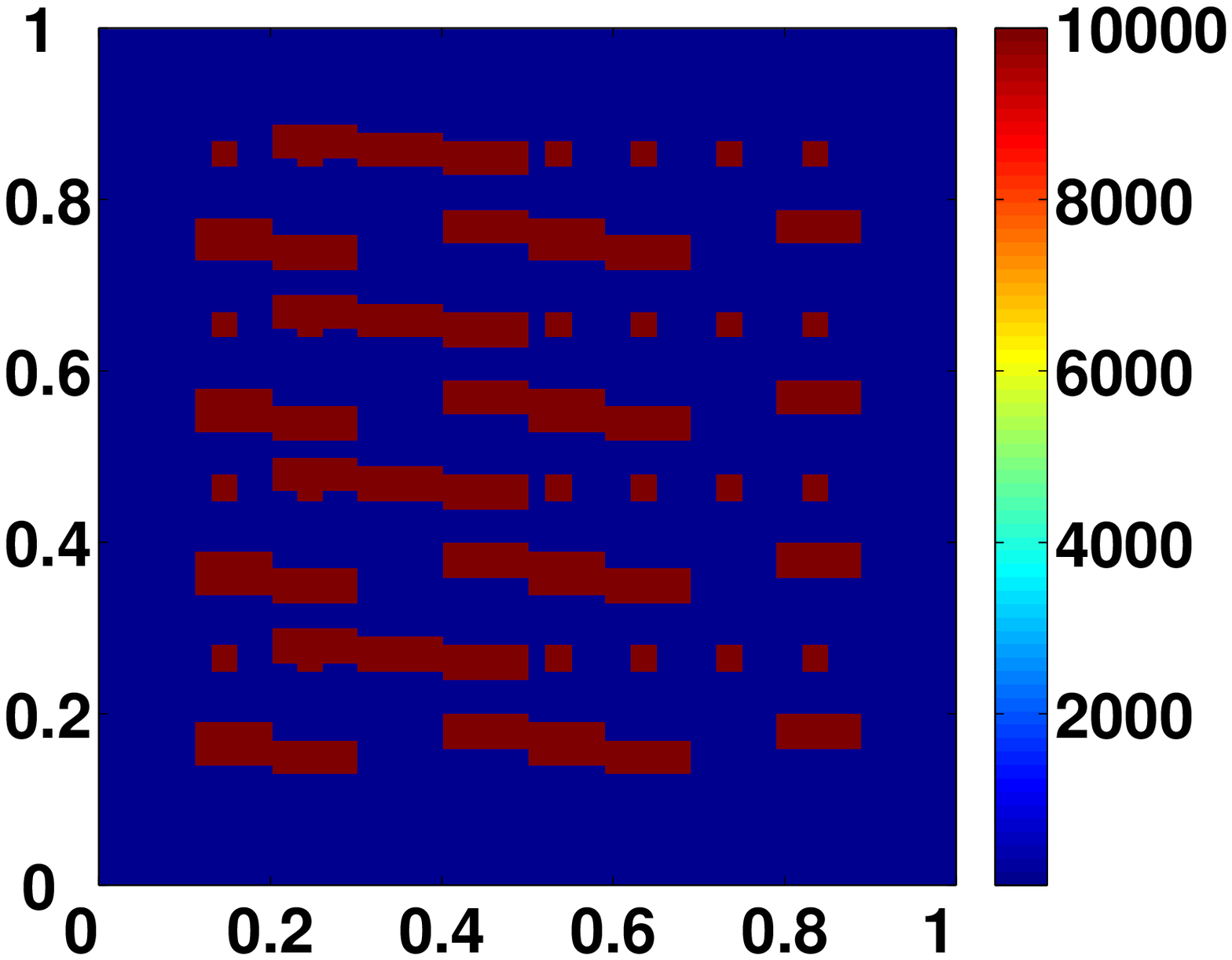}}
	\subfigure[Model 2 ($\log_{10}$ scale)]{
		\includegraphics[trim={5cm 0 5cm 0},clip,height=3.5cm]{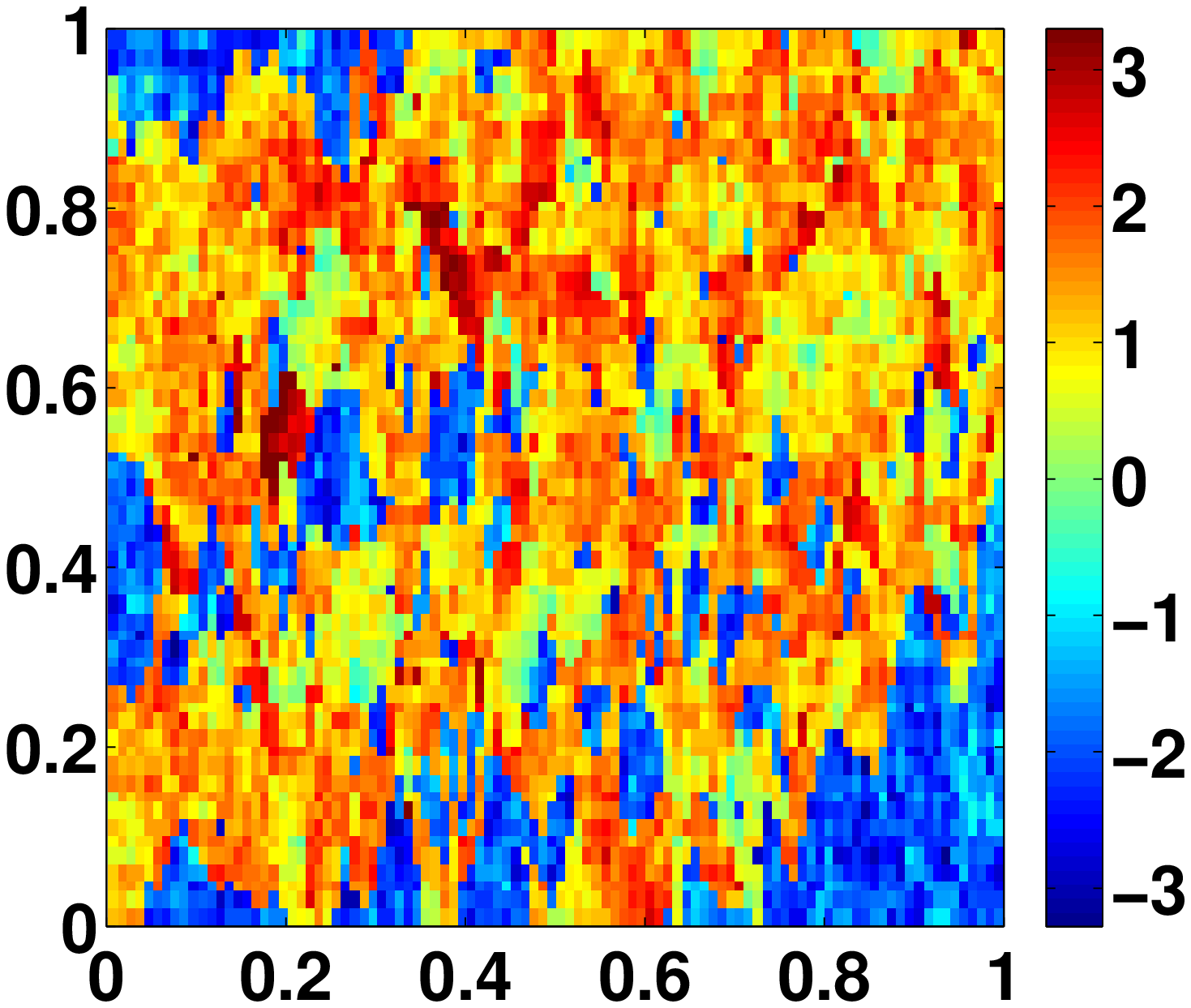}}
	\subfigure[Model 3]{
		\includegraphics[trim={5cm 0 5cm 0},clip,height=3.5cm]{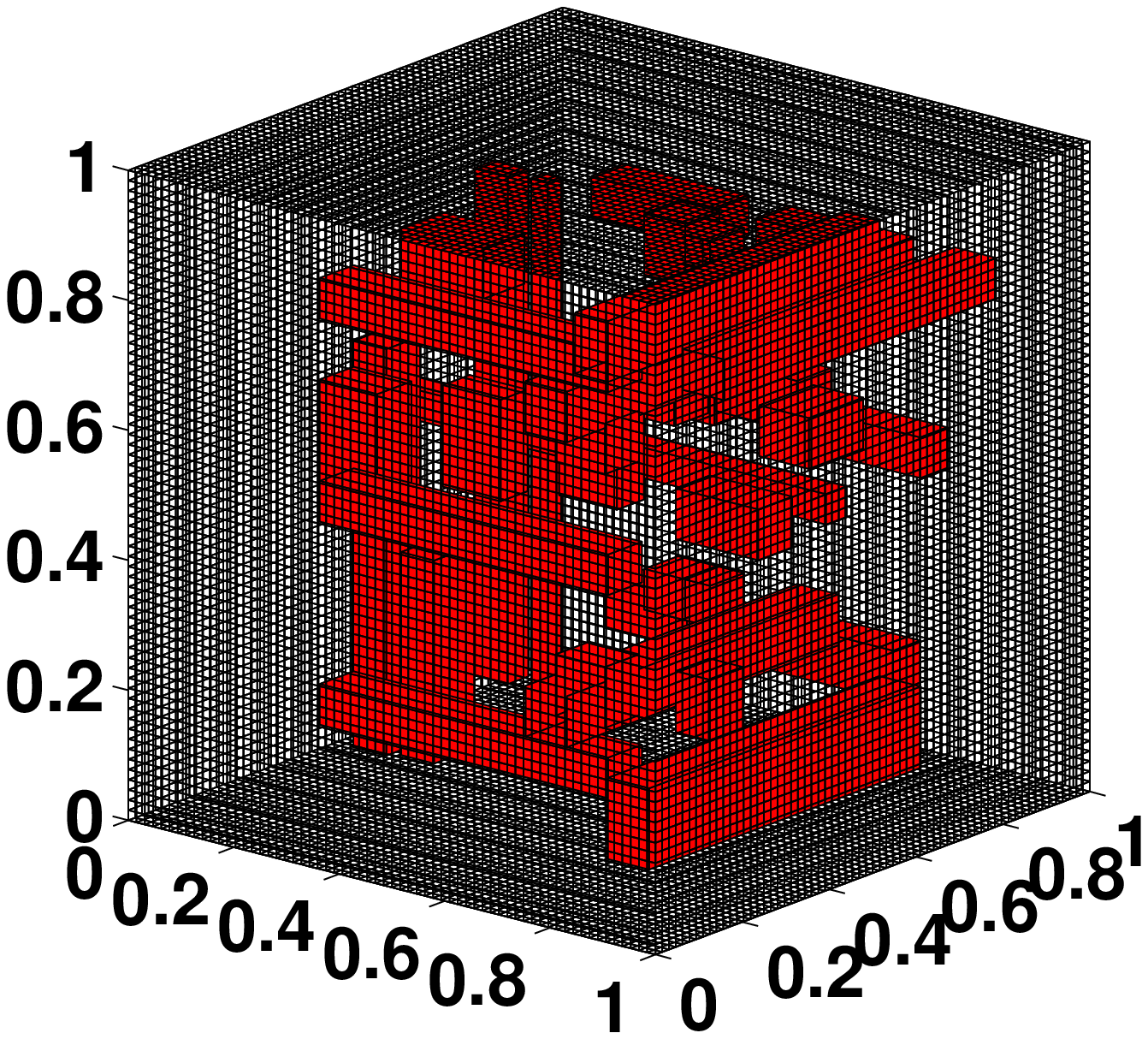}}	
	\caption{Permeability fields $\kappa$}
	\label{fig:model0} 
\end{figure}

We test our methods for the heterogeneous elliptic problem with the permeability fields $\kappa$ as depicted in Figure \ref{fig:model0}. Note that the fine scale $h$ can resolve these permeability fields. In specific, the second permeability field is the projection of the $85^{\text{th}}$ layer of the tenth SPE comparative solution project (SPE 10), cf. \cite{Aarnes2005257}, onto the fine mesh $\mathcal{T}_{h}$. In this manner, the fine mesh $\mathcal{T}_{h}$ can fit the microscale features in Model 2. In Model 3, we take $\kappa:=1$ in the background $\kappa:=10^4$ in the
red region.

In addition, to quantify the accuracy of the multiscale solutions obtained from our proposed methods, namely, ESMsFEM and WEMsFEM, we define
relative weighted $L^2$ error and energy error as follows:
\begin{equation*}
e_{L^2}=\frac{||\kappa^{1/2}(u_{ms}-u_h)||_{L^2(D)}}{||\kappa^{1/2} u_h||_{L^2(D)}},\quad
e_{H^1}=\sqrt{\frac{a(u_{ms}-u_h,u_{ms}-u_h)}{a(u_h,u_h)}}.
\end{equation*}
Recall that $u_h$ is the fine-scale solution in the finite element space $V_h$ derived from conforming Galerkin scheme, cf. \eqref{eqn:weakform_h}.
\subsection{Numerical tests for WEMsFEM}

Tables \ref{ta:nu1} and \ref{ta:nu1a} show the
numerical results of WEMsFEM with Haar and hierarchical bases  for the
test model 1. We range the coarse mesh size $H$ from $1/8$ to $1/64$, and the
wavelet level $\ell$ from $0$ to $2$. One can observe
that the accuracy of the WEMsFEM solution can be improved as the coarse mesh size $H$ is decreasing and wavelet level $\ell$ is enlarging. When the wavelet level $\ell=0$, we observe that WEMsFEM based on the Haar wavelets outperforms that based on
the hierarchical bases , while the opposite scenario occurs in the case when $\ell=1,2$.
\begin{table}[htb]
	\centering
		\begin{tabular}{|c|c|c|c|c|c|c|}
			\hline
			\multirow{2}{*}{$H$} & \multicolumn{2}{c|}{$\ell=0$} & \multicolumn{2}{c|}{$\ell=1$} & \multicolumn{2}{c|}{$\ell=2$}\tabularnewline
			\cline{2-7}
			& \specialcell{$e_{L^2}$} & \specialcell{$e_{H^1}$} & \specialcell{$e_{L^2}$} & \specialcell{$e_{H^1}$} & \specialcell{$e_{L^2}$} & \specialcell{$e_{H^1}$}\tabularnewline
			\hline
			1/8&13.81 \% &28.37\% &3.52 \%&14.18\% &0.31 \%&4.11\%  \tabularnewline
			\hline
			1/16 & 6.44\% &19.07\% &0.26 \%&4.74\% &0.05\%&2.43\%   \tabularnewline
			\hline
			1/32&2.31\% &13.15\% &0.17\%&3.80\% &0.03\%&1.79\% \tabularnewline
			\hline
			1/64 &0.86\% &8.16\% &0.08\%&2.61\% &0.01\%&0.95\%   \tabularnewline
			\hline
		\end{tabular}
	\caption{Convergence history of WEMsFEM based on Haar wavelets for Problem \eqref{eqn:pde} with Model 1.}
	\label{ta:nu1}	
\end{table}

\begin{table}[htb]
	\centering
	\begin{tabular}{|c|c|c|c|c|c|c|}
		\hline
		\multirow{2}{*}{$H$} & \multicolumn{2}{c|}{$\ell=0$} & \multicolumn{2}{c|}{$\ell=1$} & \multicolumn{2}{c|}{$\ell=2$}\tabularnewline
		\cline{2-7}
		& \specialcell{$e_{L^2}$} & \specialcell{$e_{H^1}$} & \specialcell{$e_{L^2}$} & \specialcell{$e_{H^1}$} & \specialcell{$e_{L^2}$} & \specialcell{$e_{H^1}$}\tabularnewline
		\hline
		1/8& 7.84\% &22.96\% & 2.58\%&11.94\% & 0.20\%&3.59\%  \tabularnewline
		\hline
		1/16 & 6.39\% &19.72\% & 0.73\%&6.62\% & 0.03\%&1.82\%  \tabularnewline
		\hline
		1/32&5.03\% &17.02\% & 0.22\%&4.05\% & 0.01\%&0.85\% \tabularnewline
		\hline
		1/64 &1.50\% &10.40\% & 0.05\%&1.89\% & 0.0024\%&0.42\%  \tabularnewline
		\hline
	\end{tabular}
	\caption{Convergence history of WEMsFEM based on hierarchical bases  for Problem \eqref{eqn:pde} with Model 1.}
	\label{ta:nu1a}	
\end{table}

\begin{figure}[!ht]
	\centering
	\subfigure[Reference solution]{
		\includegraphics[trim={5cm 0 5cm 0},clip,height=3.5cm]{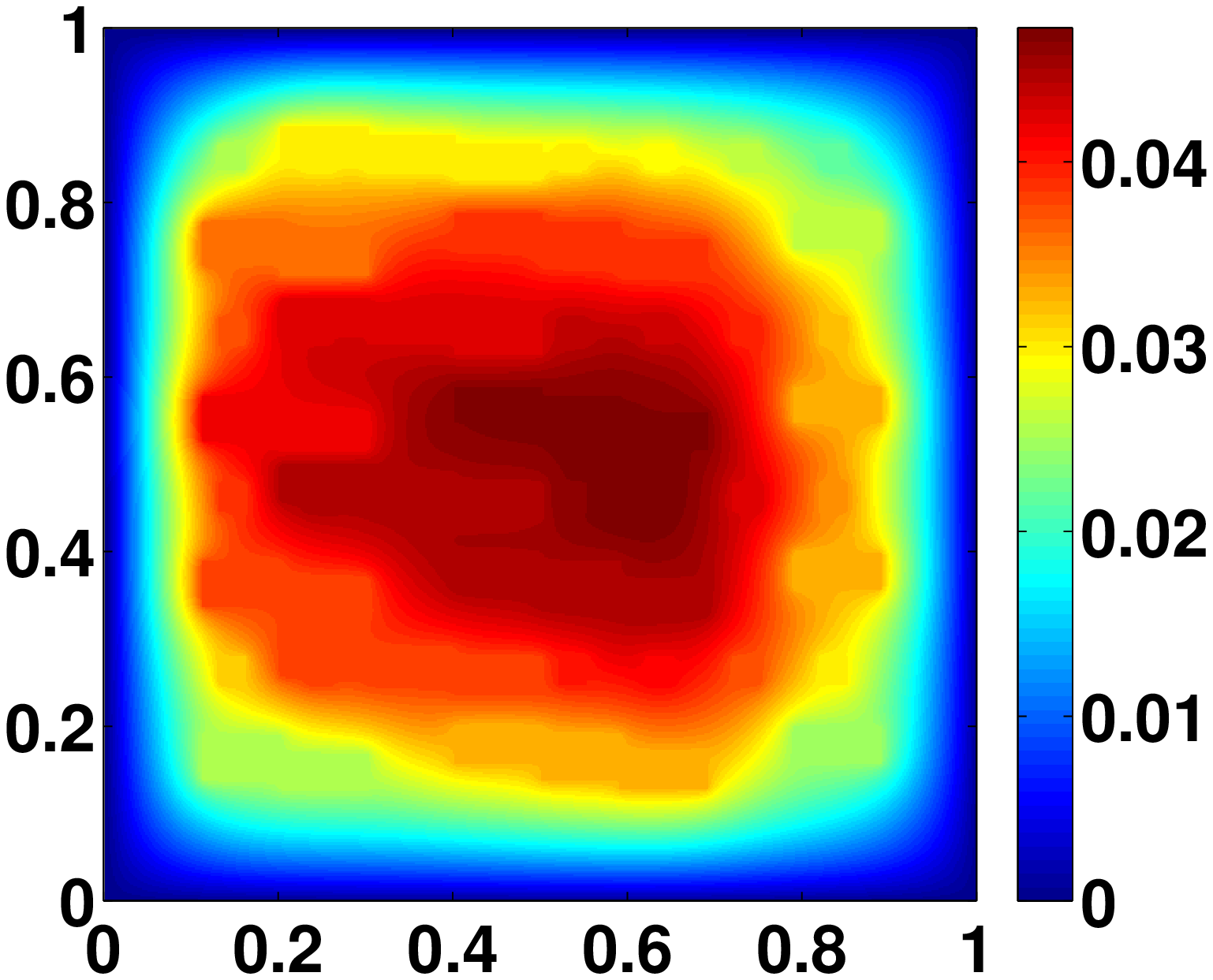}}
	\subfigure[WEMsFEM solution, $\ell=0$]{
		\includegraphics[trim={5cm 0 5cm 0},clip,height=3.5cm]{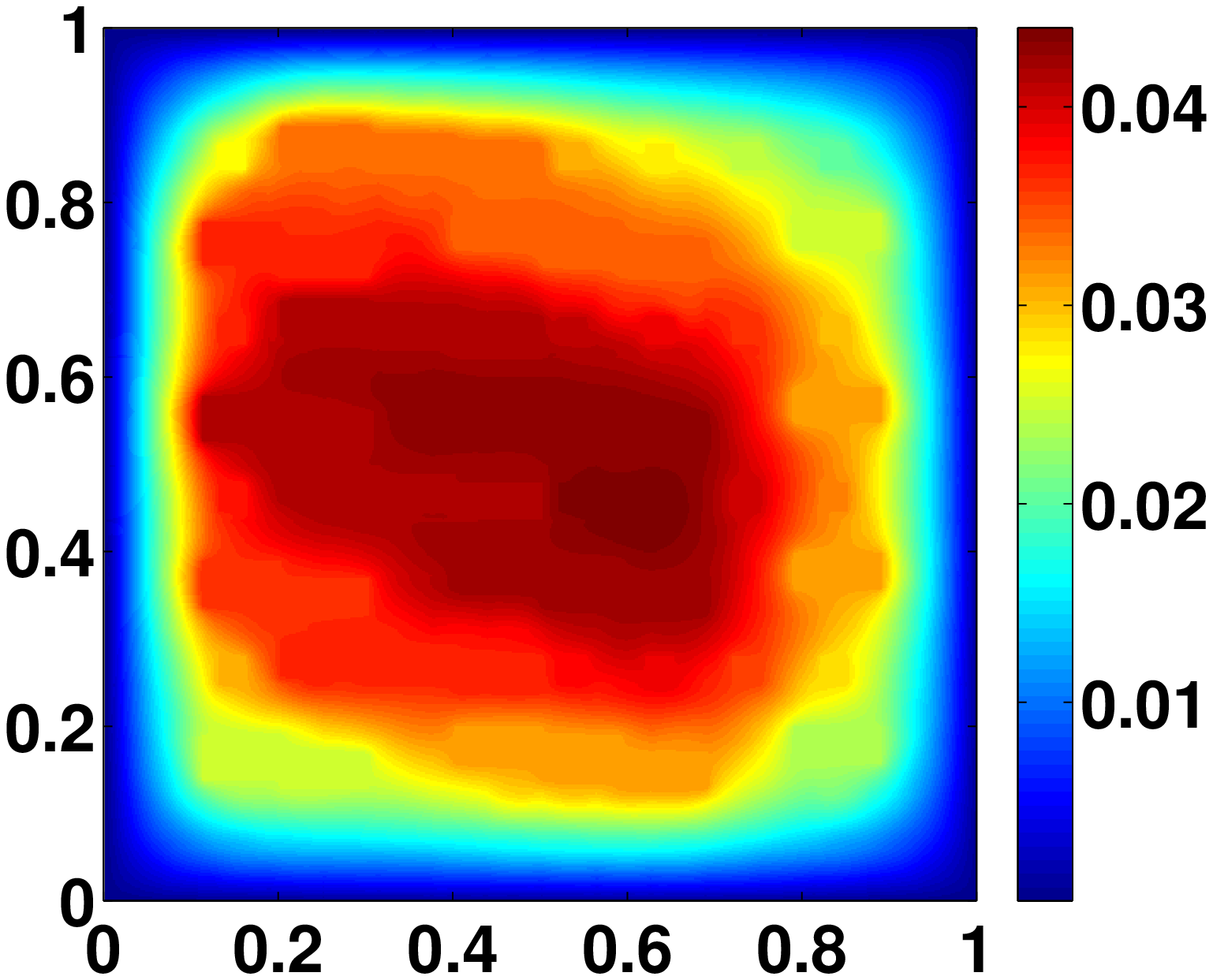}}
	\subfigure[WEMsFEM solution, $\ell=1$]{
		\includegraphics[trim={5cm 0 5cm 0},clip,height=3.5cm]{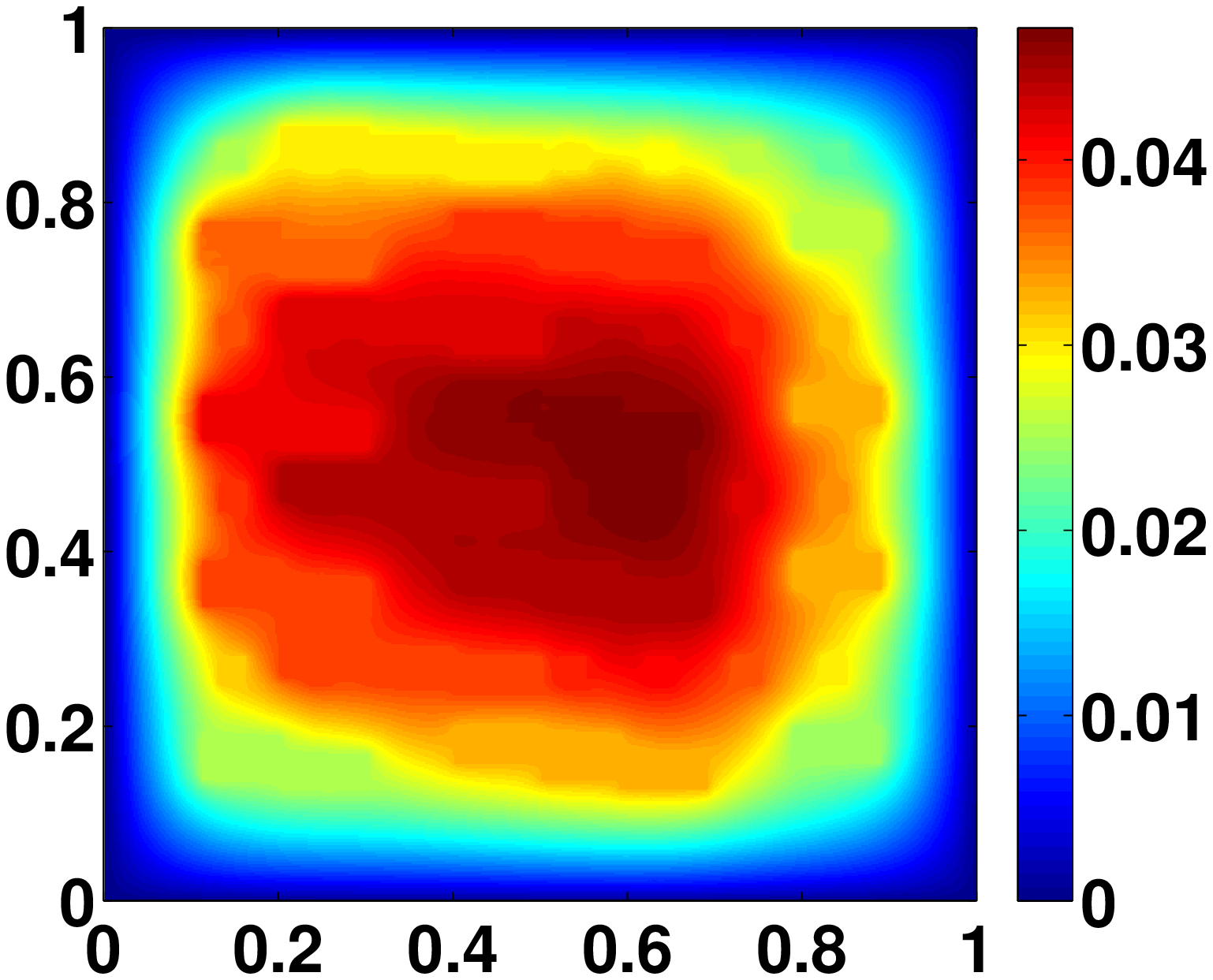}}
	\caption{The reference solution and the multiscale solutions with $H=1/16$, obtained from WEMsFEM based on Haar wavelets with levels $\ell=0$ and $1$ for Problem \eqref{eqn:pde} with Model 1.}
	\label{fig:2d}
\end{figure}

We depict in Figure \ref{fig:2d} the reference solution, the multiscale solutions solved by WEMsFEM based on Haar wavelets with the level $\ell=0$ and $1$ for Problem \eqref{eqn:pde} with Model 1. When $\ell=0$, the multiscale solution fails to capture the microscale features introduced by the complicated heterogeneity in Model 1. Nevertheless, the multiscale solution with wavelet level $\ell= 1$ is sufficient to generate a good approximation to the reference solution.
\begin{table}[H]
	\centering
		\begin{tabular}{|c|c|c|c|c|c|c|}
			\hline
			\multirow{2}{*}{$H$} & \multicolumn{2}{c|}{$\ell=0$} & \multicolumn{2}{c|}{$\ell=1$} & \multicolumn{2}{c|}{$\ell=2$}\tabularnewline
			\cline{2-7}
			& \specialcell{$e_{L^2}$} & \specialcell{$e_{H^1}$} & \specialcell{$e_{L^2}$} & \specialcell{$e_{H^1}$} & \specialcell{$e_{L^2}$} & \specialcell{$e_{H^1}$}\tabularnewline
			\hline
			1/8&3.82\% &41.98\% &2.19\%&34.46\% &1.17\%&25.17\%  \tabularnewline
			\hline
			1/16 & 2.70\% &36.24\% &1.01\%&25.51\% &0.41\%&16.88\%   \tabularnewline
			\hline
			1/32&1.32\% &24.11\% &0.33\%&14.81\% &0.13\%&10.17\% \tabularnewline
			\hline
			1/64 &0.85\% &16.95\% &0.14\%&8.81\% &0.04\%&5.75\%   \tabularnewline
			\hline
		\end{tabular}
	\caption{Convergence history of WEMsFEM based on Haar wavelets for Problem \eqref{eqn:pde} with Model 2.}
	\label{ta:nu2}
\end{table}
Furthermore, we present in Tables \ref{ta:nu2} and \ref{ta:nu2a} the convergence history of WEMsFEM for Problem \eqref{eqn:pde} with Model 2 based on Haar wavelets and hierarchical bases, respectively. Similar convergence behavior as in Tables \ref{ta:nu1} and \ref{ta:nu1a} for Model 1 can be observed. Due to limited computational resources, we only test the WEMsFEM for Model 3 with wavelets level of $\ell=0$. Its convergence history is depicted in Table \ref{ta:nu3}. As expected, the resulted multiscale solutions are not sufficiently accurate.
\begin{table}[!ht]
	\centering
	\begin{tabular}{|c|c|c|c|c|c|c|}
		\hline
		\multirow{2}{*}{$H$} & \multicolumn{2}{c|}{$\ell=0$} & \multicolumn{2}{c|}{$\ell=1$} & \multicolumn{2}{c|}{$\ell=2$}\tabularnewline
		\cline{2-7}
		& \specialcell{$e_{L^2}$} & \specialcell{$e_{H^1}$} & \specialcell{$e_{L^2}$} & \specialcell{$e_{H^1}$} & \specialcell{$e_{L^2}$} & \specialcell{$e_{H^1}$}\tabularnewline
		\hline
		1/8& 4.18\% &48.78\% &1.86 \%&33.20\% & 1.11\%&25.04\%  \tabularnewline
		\hline
		1/16 & 2.65\% &41.20\% & 1.03\%&26.44\% &0.37 \%&16.88\%  \tabularnewline
		\hline
		1/32&1.59\% &30.50\% & 0.28\%&19.50\% & 0.12\%&10.09\% \tabularnewline
		\hline
		1/64 &0.82\% &20.54\% & 0.10\%&9.29\% &0.04 \%&5.59\%  \tabularnewline
		\hline
	\end{tabular}
	\caption{Convergence history of WEMsFEM based on hierarchical bases  for Problem \eqref{eqn:pde} with Model 2.}
	\label{ta:nu2a}	
\end{table}

\begin{table}[!ht]
	\centering
	\begin{tabular}{|c|c|c|c|c|c|c|}
		\hline
		\multirow{2}{*}{$H$} & \multicolumn{2}{c|}{Haar wavelets}& \multicolumn{2}{c|}{hierarchical bases  }  \tabularnewline
		\cline{2-5}
		& \specialcell{$e_{L^2}$} & \specialcell{$e_{H^1}$}
		& \specialcell{$e_{L^2}$} & \specialcell{$e_{H^1}$} \tabularnewline
		\hline
		1/8&6.85\%& 20.41\% &9.52\%& 28.5\%  \tabularnewline
		\hline
		1/16 &9.04\%&21.50\% &9.45\%& 22.4\% \tabularnewline
		\hline
	\end{tabular}
\caption{Convergence history of WEMsFEM based on hierarchical bases  and hierarchical bases  with level of $\ell=0$ for Problem \eqref{eqn:pde} with Model 3.}
	\label{ta:nu3}
\end{table}

\subsection{Numerical tests for ESMsFEM}
In these numerical tests, we take the same number of local multiscale spectral basis functions $N_b\in \mathbb{N}_{+}$ for each coarse neighborhood $\omega_i$, where $i\in\{1,\cdots, N\}$ denotes the coarse grid index. Recall that $N\in \mathbb{N}_{+}$ is the total number of coarse grids in the coarse mesh $\mathcal{T}_{H}$. Let $\Lambda$ be the the minimum of the eigenvalues corresponding to the first eigenfunction defined in Algorithm \ref{algorithm:spectral}, which are not included in the multiscale space $V_{\text{off}}^{\rm ES} $:
\[
\Lambda=\min_{i=1,\cdots,N}\Big\{\lambda_{N_b+1}^{\ti}\Big\}.
\]

We depict the reference solution and the multiscale solutions obtained from the ESMsFEM scheme with $H=1/8$ and $N_b=2$ and $8$ for Model 3 in Figure \ref{fig:3d}.
One can conclude that the multiscale solution from ESMsFEM with $N_b=8$ is sufficient to characterize the microscale features hidden in Model 3.

\begin{figure}[H]
	\centering
	\subfigure[Reference solution]{
		\includegraphics[trim={8cm 1cm 5cm 3cm},clip,height=3cm]{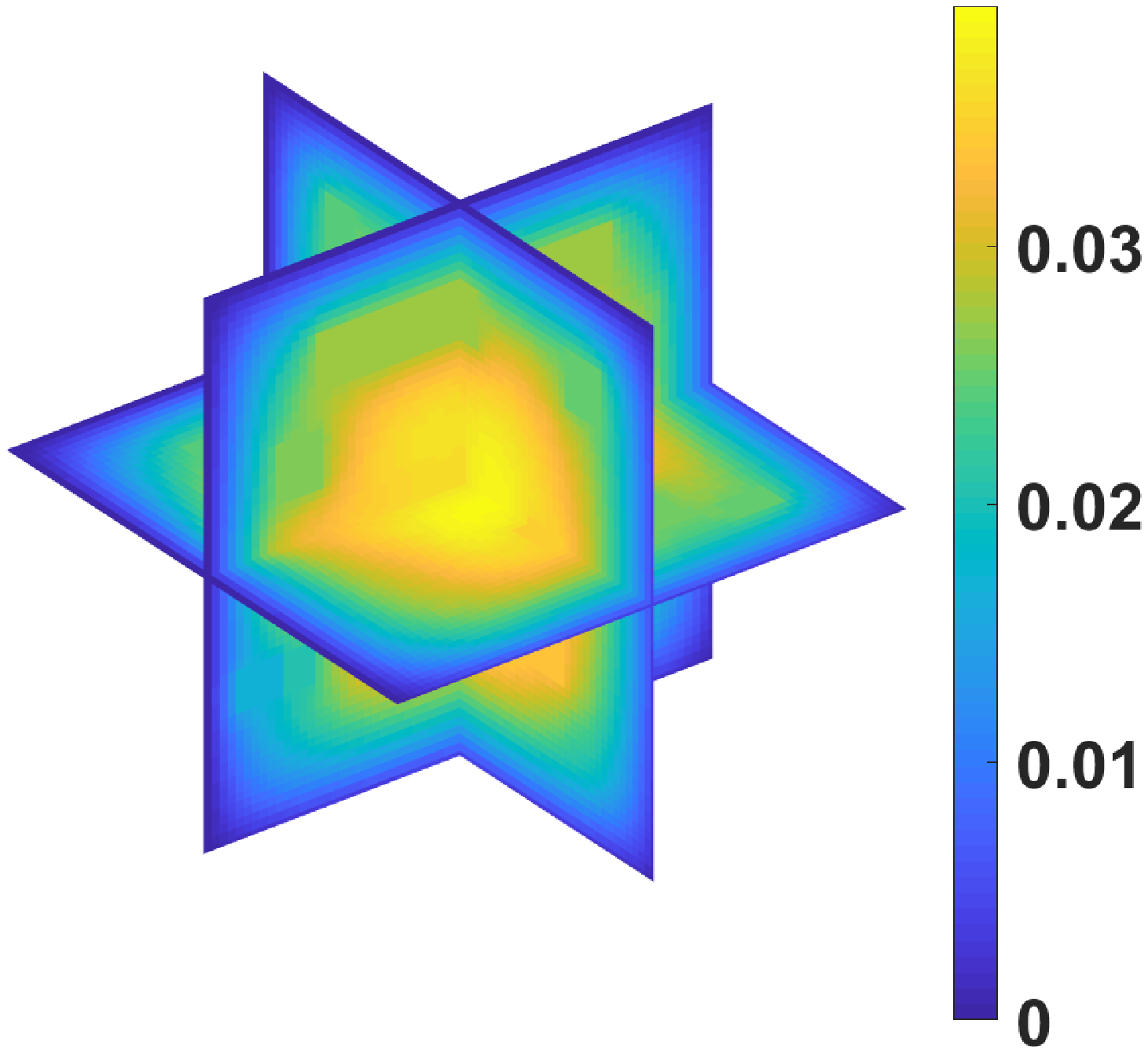}}
	\subfigure[ESMsFEM solution: $N_b=2$.]{
		\includegraphics[trim={8cm 1cm 5cm 3cm},clip,height=3cm]{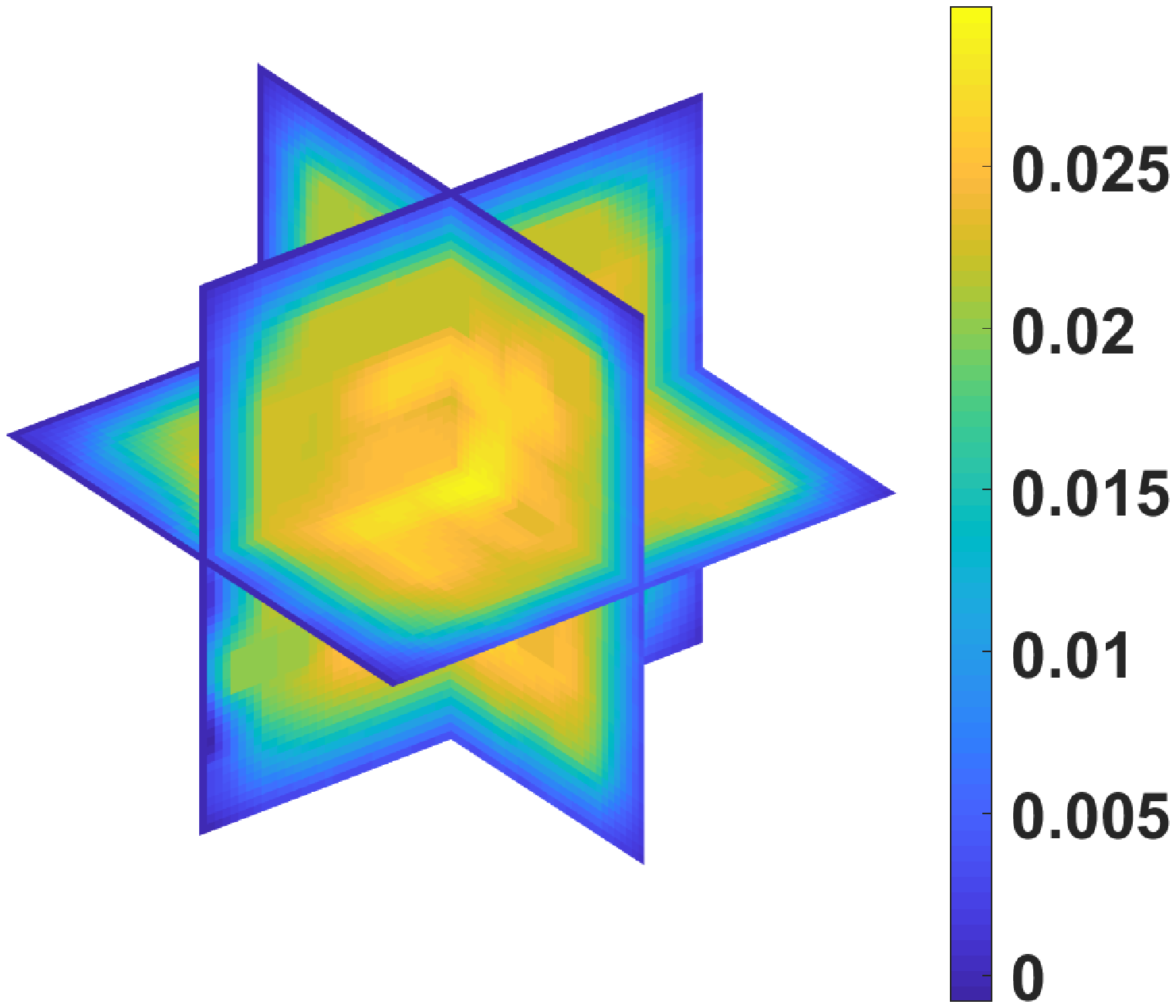}}
	\subfigure[ESMsFEM solution: $N_b=8$.]{
		\includegraphics[trim={8cm 1cm 5cm 3cm},clip,height=3cm]{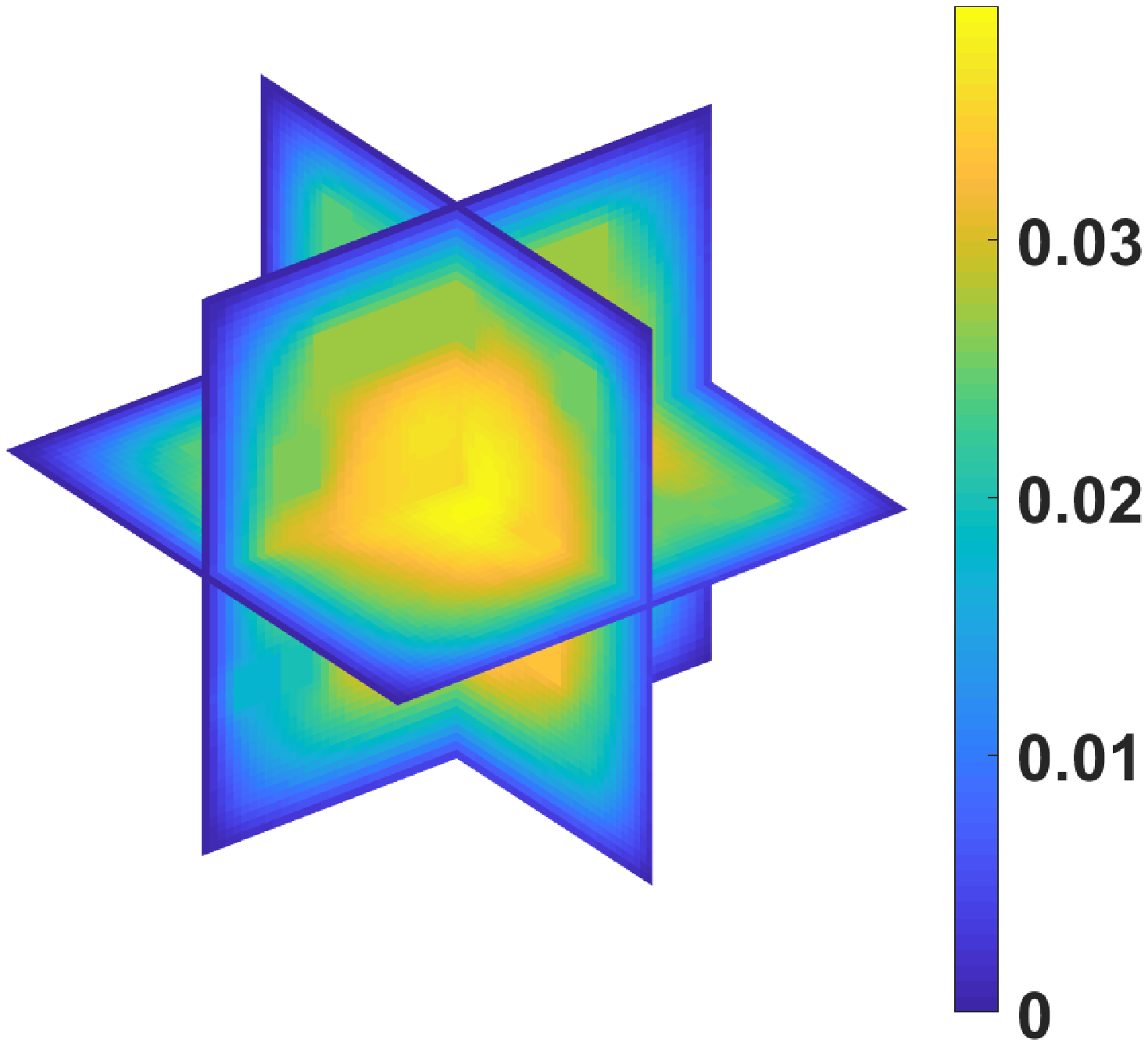}}	
	\caption{The reference solution and the ESMsFEM solutions with $N_b=2$ and $8$ for Model 3 and $H = 1/8$.}
	\label{fig:3d}
\end{figure}

The convergence history of the edge spectral multiscale method (ESMsFEM) for Problem \eqref{eqn:pde} with Models 1,2 and 3 are presented in Tables \ref{ta:convb11}-\ref{ta:convb32}. As proved in Proposition \ref{prop:edgespectral},
the multiscale solution solved by ESMsFEM converges as $\Lambda$ increases and the coarse mesh size $H$ decreases. We take Model 1 for an instance. Let $H:=1/32$, then the $L^2$ relative error decays from $3.82\%$ to $0.02\%$ as the number of bases $N_b$ increases from 2 to 10. As expected, ESMsFEM works better for model 1 compared with model 2 due to the high heterogeneity in model 2, see Tables \ref{ta:convb11}, \ref{ta:convb12}, \ref{ta:convb21} and \ref{ta:convb22}.

\begin{table}[H]
	\centering \begin{tabular}{|c|c|c|c|c|c|}\hline
		$N_b$ &$\Lambda$  & $e_{L^2}$   & $e_{H^1}$   \tabularnewline\hline
		2&74.7&3.82\%&16.36\%	\tabularnewline\hline
		4&186.4&0.37\%&5.95\%	\tabularnewline\hline
		6&347.6&0.21\%&4.53\%	\tabularnewline\hline
		8&529.4&0.05\%&2.25\%	\tabularnewline\hline
		10&743.2&0.02\%&1.43\%	\tabularnewline\hline%
	\end{tabular}
	\caption{Convergence history of ESMsFEM for Problem \eqref{eqn:pde} with Model 1 and $H=1/32$.}
	\label{ta:convb11}
\end{table}

\begin{table}[H]
\centering \begin{tabular}{|c|c|c|c|c|c|}\hline
$N_b$ &$\Lambda$  & $e_{L^2}$   & $e_{H^1}$   \tabularnewline\hline
2&440783.8&1.13\%&9.26\%	\tabularnewline\hline
4&640000.0&0.13\%&3.36\%	\tabularnewline\hline
6&1496214.7&0.08\%&2.74\%	\tabularnewline\hline
8&1537887.8&0.01\%&0.95\%	\tabularnewline\hline
10&2556030.5&0.003\%&0.50\%	\tabularnewline\hline%
\end{tabular}
\caption{Convergence history of ESMsFEM for Problem \eqref{eqn:pde} with Model 1 and $H=1/64$.}
\label{ta:convb12}
\end{table}

\begin{table}[H]
	\centering \begin{tabular}{|c|c|c|c|c|c|}\hline
		$N_b$ &$\Lambda$  & $e_{L^2}$   & $e_{H^1}$   \tabularnewline\hline
		2&1558.2&16.33\%&35.72\%	\tabularnewline\hline
		4&3760.2&10.29\%&26.85\%	\tabularnewline\hline
		6&5493.2&8.58\%&24.16\%	\tabularnewline\hline
		8&8195.6&7.33\%&22.45\%	\tabularnewline\hline
		10&9772.0&6.60\%&21.23\%	\tabularnewline\hline%
	\end{tabular}
	\caption{Convergence history of ESMsFEM for Problem \eqref{eqn:pde} with Model 2 and $H=1/32$.}
	\label{ta:convb21}
\end{table}

\begin{table}[H]
	\centering \begin{tabular}{|c|c|c|c|c|c|}\hline
		$N_b$ &$\Lambda$  & $e_{L^2}$   & $e_{H^1}$   \tabularnewline\hline
		2&47812.7&7.91\%&26.97\%	\tabularnewline\hline
		4&86609.2&4.02\%&18.00\%	\tabularnewline\hline
		6&116963.8&2.77\%&14.58\%	\tabularnewline\hline
		8&187984.5&2.08\%&12.71\%	\tabularnewline\hline
		10&212675.7&1.63\%&11.46\%	\tabularnewline\hline%
	\end{tabular}
	\caption{Convergence history of ESMsFEM for Problem \eqref{eqn:pde} with Model 2 and $H=1/64$.}
	\label{ta:convb22}
\end{table}

We present the numerical tests for Model 3 in Tables \ref{ta:convb31} and \ref{ta:convb32} corresponding to different coarse mesh sizes of $H=1/8$ and $H=1/16$. Due to limited computational resources, the case for much finer coarse grid is not performed. Compared with the numerical results for WEMsFEM, cf. Table \ref{ta:nu3}, ESMsFEM performs much better in this case. Nevertheless, ESMsFEM involves solving local eigenvalue problems and thus has much higher computational cost than WEMsFEM.    

\begin{table}[H]
	\centering \begin{tabular}{|c|c|c|c|c|c|}\hline
		$N_b$ &$\Lambda$   & $e_{L^2}$   & $e_{H^1}$   \tabularnewline\hline
		2&8.4&17.87\%&38.92\%	\tabularnewline\hline
		4&12.6&7.58\%&23.5\%	\tabularnewline\hline
		6&15.3&1.73\%&10.6\%	\tabularnewline\hline
		8&17.9&0.96\%&7.83\%	\tabularnewline\hline
		10&27.0& 0.71\%&6.20\%	\tabularnewline\hline%
	\end{tabular}
	\caption{Convergence history of ESMsFEM for Problem \eqref{eqn:pde} with Model 3 and $H=1/8$.}
	\label{ta:convb31}
\end{table}

\begin{table}[H]
	\centering \begin{tabular}{|c|c|c|c|c|c|}\hline
		$N_b$ &$\Lambda$   & $e_{L^2}$   & $e_{H^1}$   \tabularnewline\hline
		2&28.1&13.9\%&26.3\%	\tabularnewline\hline
		4&40.6&1.95\%&11.7\%	\tabularnewline\hline
		6&56.7&0.33\%&4.80\%	\tabularnewline\hline
		8&60.8&0.12\%&3.30\%	\tabularnewline\hline
		10&76.7&0.06\%&2.30\%	\tabularnewline\hline%
	\end{tabular}
	\caption{Convergence history of ESMsFEM for Problem \eqref{eqn:pde} with Model 3 and $H=1/16$.}
	\label{ta:convb32}
\end{table}

Finally, to emphasize the accuracy of the proposed methods, we provide the performance of
the (oversampling) Multiscale Finite Element Methods (MsFEMs) in Tables \ref{ta:nu1_ov}, \ref{ta:nu2_ov} and \ref{ta:nu3_ov} for the three tested permeability fields Models 1 to 3, respectively. Here, we denote $K^+$ as the oversampled region. In the case that $K^+=K$, there is no oversampling and the local multiscale basis functions are solved on each coarse element $K$, cf. \eqref{pou}; when $K^+=K+\frac{n}{2}$, the local multiscale functions are solved in a larger domain with an extra half coarse element in each direction; when $K^+=K+n$, then the local multiscale basis functions are solved in a much larger domain with one extra coarse element in each direction.

\begin{table}[htb]
	\centering
		\begin{tabular}{|c|c|c|c|c|c|c|}
			\hline
			\multirow{2}{*}{$H$} & \multicolumn{2}{c|}{$K^+=K$} & \multicolumn{2}{c|}{$K^+=K+\frac{n}{2}$} & \multicolumn{2}{c|}{$K^+=K+n$}\tabularnewline
			\cline{2-7}
			& \specialcell{$e_{L^2}$} & \specialcell{$e_{H^1}$} & \specialcell{$e_{L^2}$} & \specialcell{$e_{H^1}$} & \specialcell{$e_{L^2}$} & \specialcell{$e_{H^1}$}\tabularnewline
			\hline
			1/8&96.96\% &98.29\% &3.92\%&3235.69\% & 3.81\%&3373.30\%  \tabularnewline
			\hline
			1/16&35.97\% &53.02\% &19.70\%&619.32\% &0.74\%&434.93\%\tabularnewline
			\hline			
			1/32 &18.59\% &36.64\% &16.45\%&90.00\% &9.29\%&82.95\%  \tabularnewline
			\hline
			1/64 &6.24\% &21.22\% &5.09\%&266.35\% &3.69\%&242.37\%  \tabularnewline
			\hline
		\end{tabular}
	\caption{Convergence history of (oversampling) MsFEM for Problem \eqref{eqn:pde} with Model 1.}
	\label{ta:nu1_ov}	
\end{table}

\begin{table}[htb]
	\centering
		\begin{tabular}{|c|c|c|c|c|c|c|}
			\hline
			\multirow{2}{*}{$H$} & \multicolumn{2}{c|}{$K^+=K$} & \multicolumn{2}{c|}{$K^+=K+\frac{n}{2}$} & \multicolumn{2}{c|}{$K^+=K+n$}\tabularnewline
			\cline{2-7}
			& \specialcell{$e_{L^2}$} & \specialcell{$e_{H^1}$} & \specialcell{$e_{L^2}$} & \specialcell{$e_{H^1}$} & \specialcell{$e_{L^2}$} & \specialcell{$e_{H^1}$}\tabularnewline
			\hline
			1/8&41.12\% & 88.20\% &13.59\%&171.20\% &12.62\%& 263.41\%  \tabularnewline
			\hline
			1/16 &38.97\% &72.05\% &9.56\%&523.75\% &13.84\%&718.14\%  \tabularnewline
			\hline
			1/32&29.54\% &61.24\% &7.87\%&436.88\% &7.24\%&379.40\%\tabularnewline
			\hline
			1/64 &16.70\% &50.51\% &3.14\%&234.27\% &2.70\%&179.93\%  \tabularnewline
			\hline
		\end{tabular}
	\caption{Convergence history of (oversampling) MsFEM for Problem \eqref{eqn:pde} with Model 2.}
	\label{ta:nu2_ov}	
\end{table}

According to the numerical results, we notice that the numerical solutions solved by (oversampling) MsFEMs result in a relatively decent approximation to the reference solution measured by weighted $L^2$ norm. However, they are far from satisfactory should they be measured in the energy norm. One observes that the utilization of oversampling technique is detrimental to the approximation in energy norm. One possible explanation lies in the nonconforming nature of the multiscale basis functions when the oversampling technique is employed.

\begin{table}[H]
	\centering
	\begin{tabular}{|c|c|c|c|c|c|c|}
		\hline
		\multirow{2}{*}{$H$} & \multicolumn{2}{c|}{$K^+=K$} & \multicolumn{2}{c|}{$K^+=K+\frac{n}{2}$} & \multicolumn{2}{c|}{$K^+=K+n$}\tabularnewline
		\cline{2-7}
		& \specialcell{$e_{L^2}$} & \specialcell{$e_{H^1}$} & \specialcell{$e_{L^2}$} & \specialcell{$e_{H^1}$} & \specialcell{$e_{L^2}$} & \specialcell{$e_{H^1}$}\tabularnewline
		\hline
		1/8&95.21\% & 97.16\% &58.94\%&11.33\% &12.62\%& 410.29\%  \tabularnewline
		\hline
		1/16 &26.42\% &42.76\% &21.41\%&306.23\% &10.89\%&162.74\%  \tabularnewline
		\hline
	\end{tabular}
	\caption{Convergence history of (oversampling) MsFEM for Problem \eqref{eqn:pde} with Model 3.}
	\label{ta:nu3_ov}	
\end{table}


\section{Conclusions}\label{sec:conclusion}
We proposed in this paper two new types of edge multiscale method in the framework of the Generalized Multiscale Finite Element Methods (GMsFEMs), with their local multiscale basis functions being defined on each coarse edge. Their theoretical convergence rates were elaborately justified in terms of the number of local multiscale basis functions, the level of the wavelets and the coarse scale mesh size. Especially, the constants appearing in the estimates are independent of the multiple scales and large deviation of values in the heterogeneous coefficients. To verify our theoretical results, extensive numerical performance for elliptic problems with high-contrast heterogeneous coefficients are demonstrated. Our new proposed algorithms opens up a new direction for multiscale methods both theoretically and numerically. Future applications include convection dominated diffusion problems and Helmholtz equations with high frequencies.

\bibliographystyle{abbrv}
\bibliography{reference}
\end{document}